\documentclass{amsart}
\usepackage[T1]{fontenc}
\usepackage[utf8x]{inputenc}

\usepackage{tikz}
\usepackage{graphicx}
\usepackage{booktabs}
\usepackage{amsmath, amssymb}
\usepackage{float}
\usepackage{subcaption}
\usepackage{caption}
\usepackage{xcolor}
\usepackage{listings}
\usepackage{booktabs}
\usepackage{placeins}   
\usepackage{comment}
\usepackage{adjustbox}

\usepackage{tcolorbox}
\tcbuselibrary{skins,breakable}

\usepackage{lmodern}

\usepackage{listings}
\newcommand{\algtitle}[1]{#1}

\usepackage{amscd,amsmath,amssymb,amsthm,amsfonts,epsfig,graphics}
\usepackage{graphicx}
\usepackage{epstopdf}
\usepackage[round,authoryear]{natbib}

\usepackage{xcolor}

\usepackage[colorlinks=true,citecolor=blue]{hyperref}
\usepackage{mathabx} 
\usepackage{hyperref}
\usepackage{subcaption}
\usepackage{comment}
\let\oldtocsection=\tocsection
\let\oldtocsubsection=\tocsubsection
\renewcommand{\tocsection}[2]{\hspace{0em}\oldtocsection{#1}{#2}}
\renewcommand{\tocsubsection}[2]{\hspace{1.8em}\oldtocsubsection{#1}{#2}}

\usepackage{mdframed} 
\usepackage{lipsum}
\usepackage{tcolorbox} 
\usepackage{comment}  

\tcbuselibrary{theorems}

\usepackage[margin=3cm]{geometry}
\usepackage{algorithm}
\usepackage{algpseudocode}
\newtcolorbox{blueheadgreenbox}[1]{%
  enhanced, breakable,
  colback=green!7,            
  colframe=green!50!black,    
  boxrule=0.8pt, arc=1mm,
  left=1.2ex, right=1.2ex, top=1.2ex, bottom=1.2ex,
  colbacktitle=blue!70!black, 
  coltitle=white,             
  fonttitle=\bfseries,
  boxed title style={sharp corners, boxrule=0pt},
  attach boxed title to top left={xshift=2mm,yshift*=-2mm},
  title={\algtitle{#1}}}
\newtheorem{theorem}{Theorem}[section]
\vspace{10pt}
\newtheorem{proposition}[theorem]{Proposition}
\newtheorem{lemma}[theorem]{Lemma}

\theoremstyle{definition}
\newtheorem{definition}[theorem]{Definition}

\vspace{10pt}
\newtheorem{assumption}[theorem]{Assumption}

\newtheorem{expl}[theorem]{Example}
\theoremstyle{remark}
\newtheorem{remark}[theorem]{Remark}

\newtcbtheorem{mytheo}{Theorem}%
{colback=green!5,colframe=black!35!black,fonttitle=\bfseries}{th}

\newcommand{\nd}{\noindent}
\newcommand{\norm}[1]{\left\lVert#1\right\rVert}

\newcommand{\mubar}{\bar{\mu}}

\begin{document}
\pagenumbering{arabic}

\title{Finite Sample Smeariness on Spheres and Modulation-aware Tests}
\vspace{-8mm}
\maketitle
\begin{center}
  \textbf{Susovan Pal} \\
  Department of Mathematics and Data Science, Vrije Universiteit Brussel (VUB) \\
  Pleinlaan 2, B-1050 Elsene/Ixelles, Belgium \\
  \texttt{susovan.pal@vub.be, susovan97@gmail.com}

\end{center}



\tableofcontents

\begin{abstract}
Directional and shape data often live on manifolds, where standard central limit theorems (CLT) and the associated Wald-type \(\chi^2\)-tests require curvature-dependent recalibration: the limiting normal covariance is modified by the Hessian of the Fr\'echet function. On spheres, under the positive-Hessian assumptions considered here, this curvature correction yields Type~I finite sample smeariness (FSS), meaning that the asymptotic modulation exceeds \(1\). This paper studies FSS on \(\mathbb{S}^m\) and develops modulation-aware tests.

We first show that, for absolutely continuous distributions on \(\mathbb{S}^m\) with positive definite Hessians, Type~I FSS is unavoidable under the stated classical CLT and moment assumptions. The geometric mechanism is that, for the absolutely continuous spherical distributions considered here, positive curvature of the sphere gives the strict spectral bound \(H \prec 2I\), so that the limiting modulation is strictly larger than one.

We then prove a curse-of-dimensionality result for dimensionally comparable rotationally symmetric families: their asymptotic modulation is monotone increasing with dimension and can become arbitrarily large as the support approaches a hemisphere. Motivated by these results, we derive \textit{explicit} consistent plug-in estimators of the Fr\'echet-function Hessian and use them to construct one-sample and two-sample Hessian-corrected Hotelling-type statistics with \(\chi^2\) limits. Numerical experiments on low- and high-dimensional spheres show that the proposed modulation-aware tests have rejection behavior comparable to bootstrap-based FSS corrections, while being substantially faster due to the explicit Hessian formulas.
\end{abstract}


\section{Introduction}\label{scn:introduction}

\subsection{Fr\'echet means and Wald-type tests}
\label{subsec:intro-frechet-wald}

Directional and shape data often take values in nonlinear spaces, such as spheres, shape spaces, and more general Riemannian manifolds.  In such settings the Fr\'echet mean plays the role of the Euclidean expectation.  If \(X\) is a random variable on a Riemannian manifold \(M\), with law \(\mu\) and Riemannian distance \(d\), its Fr\'echet function is
\[
F_\mu(p):=\mathbb E[d^2(p,X)],\qquad p\in M.
\]
A minimizer of \(F_\mu\) is called a population Fr\'echet mean.  Given i.i.d.\ observations \(X_1,\dots,X_n\sim X\), the empirical Fr\'echet function is \(F_n(p):=n^{-1}\sum_{i=1}^n d^2(p,X_i)\), and a measurable minimizer \(\hat\mu_n\in\operatorname*{arg\,min}_{p\in M}F_n(p)\) is called a sample Fr\'echet mean.

Central limit theorems for sample Fr\'echet means are by now classical in statistics on manifolds; see, for instance, \citet{HendriksLandsman1998,BP2,BL,HH,H_Procrustes_10}.  The important point for inference is that, even in the ordinary \(n^{-1/2}\)-rate regime, the limiting covariance is not generally the covariance of the tangent data alone.  If \(\mubar\) is the unique population Fr\'echet mean, \(U_n:=\log_{\mubar}(\hat\mu_n)\), \(\Sigma:=\operatorname{Cov}(\log_{\mubar}X)\), and \(H:=D^2(F_\mu\circ\exp_{\mubar})(0)\), then, under the usual smoothness, uniqueness, and nondegeneracy hypotheses, the non-smeary manifold CLT has the form
\[
\sqrt n\,U_n\xrightarrow{\mathcal D}
\mathcal N(0,4H^{-1}\Sigma H^{-1}).
\]
Thus the Hessian of the Fr\'echet function enters directly into the covariance of the limiting Gaussian distribution, and hence into Wald-type \(\chi^2\)-tests based on this distribution. In Euclidean space \(H=2I\), so the correction disappears and the limiting covariance reduces to the usual covariance \(\Sigma\). On \textit{curved }spaces, however, ignoring the Hessian correction can produce systematically miscalibrated tests.  This paper studies this effect on \(\mathbb S^m\), and develops explicit Hessian-corrected, modulation-aware tests.

\subsection{Variance modulation and finite sample smeariness}
\label{subsec:intro-FSS}

Following \citet{HH2,HuckemannEltznerHundrieser2021SmearinessBegetsFSS,EltznerFSSonspheres}, we use the following formulation.

\begin{definition}[\textbf{Variance modulation and finite sample smeariness}]
\label{def:FSS}
Let \(\mubar\) be a unique Fr\'echet mean of a random variable \(X\) on a Riemannian manifold \(M\), and let \(\hat\mu_n\) be a measurable sequence of sample Fr\'echet means based on an i.i.d.\ sample of size \(n\). Define
\[
V:=F_\mu(\mubar)=\mathbb E[d^2(\mubar,X)],\qquad
V_n:=\mathbb E[d^2(\mubar,\hat\mu_n)],
\]
and, whenever \(V>0\), define the \textit{finite-sample variance modulation}
\[
m_n:=\frac{nV_n}{V}.
\]
We say that the Fr\'echet mean is \emph{finite sample smeary (FSS for short)} if
\[
1<\sup_{n\in\mathbb N}m_n<\infty.
\]
More specifically, it is called
\begin{itemize}
\item[(i)] \emph{Type~I FSS} if \(\liminf_{n\to\infty}m_n>1\);
\item[(ii)] \emph{Type~II FSS} if \(\limsup_{n\to\infty}m_n=1\).
\end{itemize}
\end{definition}

\begin{remark}\label{rmk:FSS-liminf-limsup}
Unlike in the previous literature stated here, we use \(\liminf\) and \(\limsup\) in Definition~\ref{def:FSS}, rather than assuming that the modulation sequence \(m_n\) converges.  In the standard examples where the limit exists, this formulation agrees with the usual Type~I/Type~II distinction.
\end{remark}

The Euclidean benchmark is \(m_n=1\).  Thus FSS means that there exists a finite sample size $n$ for which the modulation $m_n$ exceeds $1,$ unlike in the Euclidean setting. Type~I FSS implies that, asymptotically, the modulation stays bounded away from \(1\). Type~II FSS means that the modulation may exceed \(1\) at \textit{finite} sample sizes, but returns asymptotically to its Euclidean value. It is noteworthy that on circles, both Type~I and Type~II FSS are present; see \citet{HH2}, Theorem~3.1 and Corollary~3.2.

For later use, we also record the corresponding \textit{limiting or asymptotic modulation} in the \textit{non-smeary} CLT regime (cf. Definition \ref{def:smeariness-directional}).

\begin{definition}[\textbf{Asymptotic modulation}]
\label{def:asymptotic_modulation}
Let \(\mubar\) be a unique Fr\'echet mean of \(X\) on \(M\), and let \(\hat\mu_n\) be sample Fr\'echet means. Set \(V:=\mathbb E[d^2(\mubar,X)]\), \(V_n:=\mathbb E[d^2(\mubar,\hat\mu_n)]\), and \(m_n:=nV_n/V\) for \(V>0\).  If \(X\) is non-smeary and \(\lim_{n\to\infty}m_n\) exists, we call \(m_\infty:=\lim_{n\to\infty}m_n\) the \emph{asymptotic modulation}.
\end{definition}

In the setting of Lemma~\ref{lem:L2_from_CLT} and Assumption \ref{assump:momnt-bound}, the non-smeary CLT and the required moment bound imply
\[
m_\infty=
\frac{\operatorname{trace}(4H^{-1}\Sigma H^{-1})}{\operatorname{trace}(\Sigma)}.
\]
This formula demonstrates that a suitable Hessian correction needs to be applied to the Euclidean setting, and thus the Wald-type tests in the spherical setting also need to be modified accordingly.

\subsection{Finite sample smeariness as an intermediate regime between Euclidean and smeary regimes}
\label{subsec:intro-smeariness}

The present paper is about finite sample smeariness in the \textit{non-smeary} regime.  Nevertheless, it is useful to recall how this differs from \textit{smeariness}.  Smeariness occurs when the Fr\'echet function is flatter than quadratic at the population Fréchet mean, leading to slower-than-\(n^{-1/2}\) asymptotics; see \citet{HotzHuckemann2015,HE,Eltzner2022GeometricalSmeariness}.  The degeneracy may occur in all tangent directions, or only along a proper subspace.  We use the following definition, adapted from the subspace formulation of \citet{Eltzner2022GeometricalSmeariness}.

\begin{definition}[{\textbf{Smeariness and directional smeariness}}]
\label{def:smeariness-directional}
Let \(\mubar\) be a Fr\'echet mean of a random variable \(X\) on a Riemannian manifold \(M\), and let \(\widetilde{F}_\mu:=F_\mu \circ \exp_{\mubar}\) be the lifted Fr\'echet function in exponential coordinates at \(\mubar\), viewed on \(T_{\mubar}M\). Let \(V\subset T_{\mubar}M\) be a linear subspace.

In the smooth finite-order setting considered here, we say that \(\mu\), or equivalently \(X\), is \emph{smeary along \(V\)} if \(\widetilde{F}_\mu(x)>\widetilde{F}_\mu(0)\) for all sufficiently small \(x\in V\setminus\{0\}\), the quadratic term of \(\widetilde{F}_\mu|_V\) vanishes at \(0\), and the first nonzero term in the Taylor expansion of \(\widetilde{F}_\mu|_V-\widetilde{F}_\mu(0)\) has order \(\kappa>2\) and is positive on \(V\setminus\{0\}\). We call \(\kappa\) the\textit{ smeariness exponent }along \(V\). If \(V=T_{\mubar}M\), we say that \(\mu\) is \emph{smeary}; if \(V\) is a proper nonzero subspace, we say that \(\mu\) is \emph{directionally smeary along \(V\)}.
\end{definition}

Finite sample smeariness is different.  It is \textit{not} the regime of slower asymptotic rate.  Rather, it is the intermediate regime in which the usual \(n^{-1/2}\)-rate CLT (also called BP-CLT, after \citet{BhattacharyaPatrangenaru2003}) still holds, but the Euclidean variance calibration is still wrong.

This distinction can be made precise under the upgrade from convergence in distribution to convergence of second moments, under Assumption \ref{assump:momnt-bound}.  Suppose, for example, that in some nonzero tangent direction \(e_i\), a smeary CLT gives \(n^{1/(2(r_i-1))}\langle U_n,e_i\rangle \xrightarrow{\mathcal D} Y_i\), where \(r_i>2\) and \(\mathbb E[Y_i^2]>0\). Under Assumption \ref{assump:momnt-bound}, the squares are uniformly integrable and
\[
n\,\mathbb E\|U_n\|^2
\ge
n^{(r_i-2)/(r_i-1)}
\,\mathbb E\!\left[
\bigl(n^{1/(2(r_i-1))}\langle U_n,e_i\rangle\bigr)^2
\right]
\to\infty.
\]
Since \(d^2(\mubar,\hat\mu_n)=|U_n|^2\) whenever \(\hat\mu_n\) lies in the normal neighborhood of \(\mubar\), the variance modulation \textit{diverges}.  Thus, in the smeary or directionally smeary regimes covered by such CLTs, the modulation is not merely larger than one; it becomes\textit{ unbounded}.  Finite sample smeariness is therefore a \textit{bounded, non-smeary analogue}: worse than the Euclidean case, but less singular than full or directional smeariness.

\subsection{Main results and contribution}
\label{subsec:intro-main-results}

The first main result is that Type~I FSS is unavoidable on spheres in the positive-Hessian, non-smeary CLT regime.  More precisely, Theorem~\ref{thm:typeI_hemisphere_general} shows that if \(\mu\) is absolutely continuous on \(\mathbb S^m\), \(m\ge2\), has a unique Fr\'echet mean, positive definite Hessian, satisfies the non-smeary CLT, and satisfies the moment condition in Assumption~\ref{assump:momnt-bound}, then the asymptotic modulation exists and is strictly larger than one.  

The second main result is a \textit{curse of dimensionality effect} on modulation.  For \textit{dimensionally comparable rotationally symmetric families}, introduced in Definition~\ref{def:dimensionally-comparable-distributions}, Theorem~\ref{thm:dim_monotone_modulation_rot} proves that the asymptotic modulation is strictly increasing with the sphere dimension.  It also gives an explicit high-dimensional limit of the asymptotic modulation $m_{\infty}$.  In particular, when the support approaches a hemisphere, the limiting modulation can be made arbitrarily large. 

The third contribution is \textit{inferential}.  The explicit Hessian formulas in Proposition~\ref{prop:Hessian-common} in the general-density and rotationally symmetric cases lead to consistent plug-in Hessian estimators.  Proposition~\ref{prop:Hn_consistency} proves consistency in both the general-density and rotationally symmetric settings.  These estimators are then used to construct Hessian-corrected quadratic statistics.  Proposition~\ref{prop:quadratic_chisq} and Theorem~\ref{thm:mod-aware-test} give the one-sample version, while Theorem~\ref{thm:two-sample-mod-aware-hotelling} gives the corresponding two-sample Hotelling-type statistic.  The point is that the naive tangent covariance is replaced by the curvature-corrected covariance \(4H^{-1}\Sigma H^{-1}\).

Section~\ref{sec:numerical-bootstrap-vs-hessian} compares three procedures: a naive tangent-space \(\chi^2\)-test, a bootstrap-based finite-sample-smeariness correction, and the proposed Hessian-corrected modulation-aware test.  The low-dimensional experiment on \(\mathbb S^2\) shows that the two corrected procedures behave similarly near the null.  The high-dimensional experiment on \(\mathbb S^{36}\) shows a much sharper effect: the naive test severely over-rejects, while the bootstrap-based and modulation-aware tests both correct the level.  The runtime comparison in Table~\ref{tab:runtime} shows that, in the reported high-dimensional experiment, the modulation-aware test is about \(38\) times faster per test call.

Finally, Section~\ref{subsec:microtus-real-data-diagnostic} applies the two-sample procedure to
\textit{Microtus} lower first molar preshape data on \(\mathbb S^{39}\).  The data are highly
concentrated around their Fr\'echet mean, so the empirical Hessian is nearly \(2I\) and the
modulation-aware statistic is numerically almost identical to the naive Hotelling statistic,
providing a real-data local-Euclidean sanity check.

The sharp support-threshold theory for full and directional smeariness on spheres is treated separately in \citet{Pal2026SharpSupportThresholds}.  The present paper uses that viewpoint only as background and as a source of clean non-smeary regimes.  Its main focus is finite sample smeariness, dimension-dependent modulation, and Hessian-corrected inference.

\subsection{Organization of the paper}
\label{subsec:intro-organization}

Section~\ref{scn:TE-sqdist} derives the second-order Taylor expansion of the squared Riemannian distance on \(\mathbb S^m\) in normal coordinates.  Section~\ref{scn:common-setup-derivatives} uses this expansion to obtain the Hessian formula for the Fr\'echet function, first for general densities and then under rotational symmetry; see Proposition~\ref{prop:Hessian-common}.  Section~\ref{scn:asymptotics-b-h} records the radial function \(b_m\) and its unique zero \(R_m\), which control the rotationally symmetric Hessian; see Proposition~\ref{prop:zero_bm}.

Section~\ref{scn:FSS} proves the Type~I finite sample smeariness theorem.  Lemma~\ref{lem:L2_from_CLT} gives the required moment-convergence step, Lemma~\ref{lem:spectral_bound_Hess} proves the strict spectral bound for the Hessian, and Theorem~\ref{thm:typeI_hemisphere_general} combines these ingredients to obtain \(m_\infty>1\).  Section~\ref{scn:curse-dim-modulation} proves the curse-of-dimensionality theorem for modulation, Theorem~\ref{thm:dim_monotone_modulation_rot}, and Section~\ref{subsec:numerical-curse-dim-modulation} gives a numerical illustration of this dimension effect.

Section~\ref{scn:modulation-aware-tests} develops the statistical procedures.  Proposition~\ref{prop:Hn_consistency} proves consistency of the plug-in Hessian estimators, Proposition~\ref{prop:quadratic_chisq} gives the basic quadratic-form limit, Theorem~\ref{thm:mod-aware-test} gives the one-sample modulation-aware test, and Theorem~\ref{thm:two-sample-mod-aware-hotelling} gives the two-sample Hessian-corrected Hotelling statistic.  Section~\ref{sec:numerical-bootstrap-vs-hessian} presents the simulation study comparing naive, bootstrap-based, and modulation-aware tests.  The symbolic computation used for the Taylor expansion is recorded in Appendix~\ref{scn:symbolic-computation}.

\subsection{{Acknowledgements and funding.}}

This research was supported by funding from Research Foundation--Flanders (FWO) via the Odysseus II programme no.~G0DBZ23N. The author acknowledges technical discussions with Stephan Huckemann from the University of G\"ottingen, Germany, and also encouragement by David Tewodrose at Vrije Universiteit Brussel.

\section{Taylor expansion of the Riemannian squared distance on $\mathbb{S}^m$ through order $2$}
\label{scn:TE-sqdist}

This section records the Taylor expansion, in normal coordinates at the north pole, of the squared Riemannian distance on the sphere through order two.  This second-order expansion is the local computational input used later to derive the Hessian of the Fr\'echet function.  Throughout, the expansion is taken in the first variable while the second variable is fixed away from the cut locus.

\subsection{Taylor expansion of the Riemannian squared distance function in the tangent space using the exponential map}

Let $\mathbb{S}^{m}\subset\mathbb{R}^{m+1}$ be the unit sphere, let $N\in\mathbb{S}^{m}$ be the north pole, and identify $T_{N}\mathbb{S}^{m}\cong\mathbb{R}^{m}$ with the Euclidean inner product $\langle\cdot,\cdot\rangle$ and norm $\|\cdot\|$.  Fix $v\in T_{N}\mathbb{S}^{m}$ with \(R:=\|v\|\in(0,\pi)\), and define the pulled-back squared distance function on \(T_N\mathbb S^m\) by
\begin{equation}\label{eqn:pulled-back-sq-dist}
   g(u;v):=d^{2}\big(\exp_{N}u,\exp_{N}v\big).
\end{equation}

For $a\in T_{N}\mathbb{S}^{m}$ with $\|a\|=1$, set \(\alpha:=\langle a,v\rangle\) and
\(f_{a}(t):=g(ta;v)=d^{2}(\exp_{N}(ta),\exp_{N}v)\).  Thus \(f_a\) is the line restriction of \(g(\cdot;v)\) along \(a\).  By the spherical cosine law in normal coordinates,
\begin{equation}\label{eqn:f_a}
\cos\big(d(\exp_{N}(ta),\exp_{N}v)\big)
=
\cos t\,\cos R
+
\sin t\,\frac{\sin R}{R}\,\alpha.
\end{equation}
Equivalently, if \(d(t):=d(\exp_{N}(ta),\exp_{N}v)\in(0,\pi)\), then
\(f_a(t)=d(t)^2\), where
\(d(t)=\arccos(\cos t\,\cos R+\sin t\,(\sin R/R)\alpha)\).  Expanding at \(t=0\) gives
\(f_a(t)=\sum_{k=0}^{2}C_k(R,\alpha)t^k+O(t^3)\), with
\begin{equation}\label{eqn:C}
\boxed{
\begin{aligned}
C_{0}(R,\alpha)&=R^{2},
\qquad
C_{1}(R,\alpha)=-2\alpha,\\
C_{2}(R,\alpha)&=\frac{R}{\tan R}+\alpha^{2}\Big(\frac{1}{R^{2}}-\frac{1}{R\tan R}\Big).
\end{aligned}
}
\end{equation}
The symbolic calculation of these Taylor coefficients is recorded in Appendix~\ref{scn:symbolic-computation}.  For \(k=1,2\), whenever \(g(\cdot;v)\) is \(k\) times differentiable at \(0\), the line-restriction identity gives
\(f_a^{(k)}(0)=D^kg(0;v)[a,\dots,a]=k!C_k(R,\alpha)\).

Now write \(u\neq 0\) as \(u=ra\), with \(r:=\|u\|\) and \(\|a\|=1\).  Then \(g(ra;v)=f_a(r)\), and \(\alpha=\langle a,v\rangle=\langle u,v\rangle/\|u\|\).  Hence, for fixed \(v\), the expansion through degree \(2\) is
\begin{equation}\label{eqn:g-eqn}
\boxed{
g(u;v)
=
R^{2}
-2\langle u,v\rangle
+
\frac{R}{\tan R}\,\|u\|^{2}
+
\Big(\frac{1}{R^{2}}-\frac{1}{R\tan R}\Big)\,\langle u,v\rangle^{2}
+O(\|u\|^{3}),
\qquad R:=\|v\|.
}
\end{equation}

\subsection{Explicit formula for the Hessian of the squared distance}
\label{subsec:derivatives-g}

We next extract from \eqref{eqn:g-eqn} the second derivative of \(g(\cdot;v)\) at \(0\).  Fix \(v=R\Theta\), where \(R:=\|v\|\in(0,\pi)\) and \(\Theta\in\mathbb{S}^{m-1}\).  The derivatives below are always taken with respect to the \(u\)-variable, with \(v\) fixed.

In the Euclidean vector space \(T_N\mathbb S^m\), one has \(D^2_u\|u\|^2[w,w]=2\|w\|^2\) and \(D^2_u\langle u,v\rangle^2[w,w]=2\langle w,v\rangle^2=2R^2\langle w,\Theta\rangle^2\).  Therefore \eqref{eqn:g-eqn} gives, for every \(w\in T_N\mathbb S^m\),
\begin{equation}\label{eqn:Hess-g}
\boxed{
D_u^2 g(0;R\Theta)[w,w]
=
2\frac{R}{\tan R}\|w\|^2
+
2\Bigl(1-\frac{R}{\tan R}\Bigr)\langle w,\Theta\rangle^2.
}
\end{equation}


\section{Setup and Hessian formulas}
\label{scn:common-setup-derivatives}
\label{scn:support-rot-symm}
\label{scn:support-general-density}

This section records the common setup and the Hessian formulas for the population Fr\'echet function on \(\mathbb S^m\).  We first treat an arbitrary absolutely continuous density, not assumed to be rotationally symmetric, and then obtain the rotationally symmetric formula as a reduction.  The pointwise expansion of the squared Riemannian distance \(g(u;v)=d^2(\exp_Nu,\exp_Nv)\) is the one derived in Section~\ref{scn:TE-sqdist}.

\subsection{General densities and rotational symmetry}
\label{scn:set-up}
\label{scn:set-up-general-density}

Let \(\mu\) be an absolutely continuous probability measure on \(\mathbb S^m\).  After applying an isometry, we take a population Fr\'echet mean to be the north pole \(N\), and write \(\mubar=N\).  Since \(\operatorname{Cut}(N)=\{-N\}\) and \(\mu\) is absolutely continuous, \(\mu(\{-N\})=0\).  Thus \(V:=\exp_N^{-1}(X)\) is defined \(\mu\)-a.s., and we write \(V=R\Theta\), with \(R\in(0,\pi)\) and \(\Theta\in\mathbb S^{m-1}\).  In these coordinates, \(d\operatorname{vol}_{\mathbb S^m}=(\sin R)^{m-1}\,dR\,d\Theta\).

Assume that \(\widetilde\mu:=(\exp_N^{-1})_\#\mu\) has density \(\rho(R,\Theta)\) in polar coordinates, so that
\begin{equation}
\label{eqn:general-density-polar}
    d\widetilde\mu(R,\Theta)=\rho(R,\Theta)(\sin R)^{m-1}\,dR\,d\Theta,
    \qquad
    \int_0^\pi\!\!\int_{\mathbb S^{m-1}}\rho(R,\Theta)(\sin R)^{m-1}\,d\Theta\,dR=1 .
\end{equation}
The lifted Fr\'echet function in these coordinates is
\begin{equation}\label{eqn:Frechet-general-polar}
    \widetilde F_\mu(u):=F_\mu(\exp_Nu)
    =
    \int_0^\pi\!\!\int_{\mathbb S^{m-1}}
    g(u;R\Theta)\rho(R,\Theta)(\sin R)^{m-1}\,d\Theta\,dR .
\end{equation}
The rotationally symmetric case is the specialization \(\rho(R,\Theta)=\rho(R)\), giving
\begin{equation}
\label{eqn:Fréchet-rot-symm-double-int}
    \widetilde F_\mu(u)
    =
    \int_0^\pi\!\!\int_{\mathbb S^{m-1}}
    g(u;R\Theta)\rho(R)(\sin R)^{m-1}\,d\Theta\,dR .
\end{equation}

\begin{assumption}[\textbf{Cut-locus avoidance}]
\label{assum:density-assumption}
We assume that there exists \(\varepsilon_0>0\) such that \(\rho(R,\Theta)=0\) for a.e. \((R,\Theta)\) with \(R\in(\pi-\varepsilon_0,\pi)\).  In the rotationally symmetric case this means \(\rho(R)=0\) for \(R\) sufficiently close to \(\pi\).  Under this assumption, the integrand is smooth in \(u\) for \(u\) near \(0\) and is uniformly separated from the cut locus. Hence differentiation under the integral sign at \(u=0\), in particular up to second order, is justified; compare \citet{tran2020sampling}, Proposition 3.2.1.
\end{assumption}

\subsection{Formula for the Hessian of the Fréchet function and its rotationally symmetric reduction}
\label{para:Hess-general-density}

\begin{proposition}[\textbf{Formula for the Hessian of the Fréchet function}]
\label{prop:Hessian-common}
Assume Assumption~\ref{assum:density-assumption}.  Let \(H:=D^2\widetilde F_\mu(0)\) denote the Hessian of the lifted Fr\'echet function at \(0\in T_N\mathbb S^m\), equivalently the Hessian of \(F_\mu\) at \(\mubar=N\).  Then
\begin{equation}
\label{eqn:Hessian-general-density}
\boxed{
    H
    =
    2\int_0^\pi\!\!\int_{\mathbb S^{m-1}}
    \left[
    \frac{R}{\tan R}I_m+
    \left(1-\frac{R}{\tan R}\right)\Theta\Theta^{\mathsf T}
    \right]
    \rho(R,\Theta)(\sin R)^{m-1}\,d\Theta\,dR .
}
\end{equation}
Equivalently, for every \(w\in T_N\mathbb S^m\),
\begin{equation}
\label{eqn:Hessian-general-bilinear-form}
    w^{\mathsf T}Hw
    =
    2\int_0^\pi\!\!\int_{\mathbb S^{m-1}}
    \left[
    \frac{R}{\tan R}\|w\|^2+
    \left(1-\frac{R}{\tan R}\right)\langle w,\Theta\rangle^2
    \right]\rho(R,\Theta)(\sin R)^{m-1}\,d\Theta\,dR .
\end{equation}
In the rotationally symmetric case, define
\begin{equation}\label{eqn:b}
    b_m(R):=1+(m-1)R\cot R .
\end{equation}
Then
\begin{equation}
\label{eqn:Hess}
\boxed{
    H
    =
    \frac{2\operatorname{vol}(\mathbb S^{m-1})}{m}
    \left(\int_0^\pi \rho(R)(\sin R)^{m-1}b_m(R)\,dR\right)I_m .
}
\end{equation}
Thus, under rotational symmetry, the Hessian is singular if and only if it vanishes, equivalently if and only if the scalar integral in \eqref{eqn:Hess} vanishes.
\end{proposition}

\begin{proof}
By \eqref{eqn:Frechet-general-polar}, Assumption~\ref{assum:density-assumption}, and the pointwise Hessian formula \eqref{eqn:Hess-g}, differentiating under the integral sign gives \eqref{eqn:Hessian-general-bilinear-form}, hence \eqref{eqn:Hessian-general-density}.  If \(\rho(R,\Theta)=\rho(R)\), then the Funk--Hecke formula gives
\[
    \int_{\mathbb S^{m-1}}\langle w,\Theta\rangle^2\,d\Theta
    =
    \frac{\operatorname{vol}(\mathbb S^{m-1})}{m}\|w\|^2;
\]
see, for example, \citet{ChafaiFunkHecke}.  Therefore \eqref{eqn:Hess} follows.  The final assertion follows because the Hessian is then a scalar multiple of \(I_m\).
\end{proof}

\section{Radial threshold function \(b_m\) and its zero}
\label{scn:asymptotics-b-h}

Under rotational symmetry, the Hessian formula \eqref{eqn:Hess} shows that the lifted Fr\'echet function is governed, at second order, by the radial function \(b_m\) defined in \eqref{eqn:b}.  We record here the elementary sign structure of \(b_m\).  Its unique zero \(R_m\) marks the Hessian threshold used below for the rotationally symmetric non-smeary regime.

\subsection{Analysis of \(b_m\)}

\begin{proposition}[\textbf{The Hessian threshold}]
\label{prop:zero_bm}
Let \(m\ge2\), and following \eqref{eqn:b}, let \(b_m:(0,\pi)\to\mathbb R\) be defined by \(b_m(R)=1+(m-1)R\cot R\).  Then \(b_m\) is strictly decreasing on \((0,\pi)\), has a unique zero
\begin{equation}\label{eqn:R_m}
    R_m\in(\pi/2,\pi),
\end{equation}
and satisfies \(b_m(R)>0\) for \(R\in(0,R_m)\) and \(b_m(R)<0\) for \(R\in(R_m,\pi)\).
\end{proposition}

\begin{proof}
Set \(q(R):=R\cot R\). Then
\[
    q'(R)=\cot R-R\csc^2R=\frac{\sin R\cos R-R}{\sin^2R}.
\]
The numerator \(s(R):=\sin R\cos R-R\) satisfies \(s(0)=0\) and \(s'(R)=\cos(2R)-1<0\) on \((0,\pi)\).  Hence \(q'(R)<0\), so \(b_m\) is strictly decreasing.  Moreover, \(q(R)\to1\) as \(R\downarrow0\), while \(q(R)\to-\infty\) as \(R\uparrow\pi\).  Thus \(b_m\) has a unique zero.  Since \(b_m(\pi/2)=1>0\), this zero lies in \((\pi/2,\pi)\), and the stated signs follow from strict monotonicity.
\end{proof}


\section{Finite sample smeariness in the non-smeary regime}
\label{scn:FSS}

This section proves that, on spheres of dimension \(m\ge2\), the ordinary non-smeary central limit regime is still finite-sample smeary of Type I whenever the Hessian is positive definite and the corresponding second moments converge.  The mechanism is second-order: the Hessian correction in the Fr\'echet mean CLT inflates the limiting trace because of positive curvature.

We start with a technical lemma used to upgrade convergence in distribution to convergence of second moments.

\begin{lemma}[\textbf{Convergence of second moments under a uniform \((2+\delta)\) bound}]
\label{lem:L2_from_CLT}
Let \(Y_n\) be a sequence of \(\mathbb R^m\)-valued random variables.  Assume that \(Y_n\rightarrow_{\mathcal D}Y\) and that there exists \(\delta>0\) such that
\[
    \sup_{n\ge1}\mathbb E\|Y_n\|^{2+\delta}<\infty .
\]
Then \(\mathbb E\|Y_n\|^2\to \mathbb E\|Y\|^2\).
\end{lemma}

\begin{proof}
The moment bound implies uniform integrability of \(\{\|Y_n\|^2\}_{n\ge1}\).  The conclusion is the standard consequence of convergence in distribution together with uniform integrability; see \citet{Billingsley1999}.
\end{proof}

We shall use the following moment assumption for sample Fr\'echet means.

\begin{assumption}[\textbf{Uniform moment bound}]
\label{assump:momnt-bound}
Let \(\widehat\mu_n\) be a measurable sequence of intrinsic sample Fr\'echet means and set \(U_n:=\exp_N^{-1}(\widehat\mu_n)\).  We assume that there exists \(\delta>0\) such that
\[
    \sup_{n\ge1}\mathbb E\bigl[\|\sqrt n\,U_n\|^{2+\delta}\bigr]<\infty .
\]
\end{assumption}

\begin{lemma}[\textbf{Strict spectral upper bound for the Hessian}]
\label{lem:spectral_bound_Hess}
Let \(\mu\) be absolutely continuous on \(\mathbb S^m\), \(m\ge2\), with unique Fr\'echet mean \(N\), and assume that \(\mu\) satisfies the setup of Section~\ref{scn:set-up-general-density} and Assumption~\ref{assum:density-assumption}.  Let \(H:=D^2\widetilde F_\mu(0)\) and set \(M:=\frac12 H\).  Then \(\lambda_{\max}(M)<1\).
\end{lemma}

\begin{proof}
Let \(a(R):=R/\tan R\).  As in the proof of Proposition~\ref{prop:zero_bm}, \(a(R)<1\) for \(R\in(0,\pi)\).  For \(\|w\|=1\), the Hessian formula \eqref{eqn:Hessian-general-bilinear-form} gives
\begin{equation}\label{eq:lmaxM_strict}
\begin{aligned}
w^{\mathsf T}Mw
&=\int_0^\pi\!\!\int_{\mathbb S^{m-1}}
\Big(a(R)+(1-a(R))\langle w,\Theta\rangle^2\Big)
\rho(R,\Theta)(\sin R)^{m-1}\,d\Theta\,dR\\
&\le
\int_0^\pi\!\!\int_{\mathbb S^{m-1}}
\Big(a(R)+(1-a(R))\Big)
\rho(R,\Theta)(\sin R)^{m-1}\,d\Theta\,dR
=1 .
\end{aligned}
\end{equation}
Since \(\rho\) is a density with respect to \((\sin R)^{m-1}dR\,d\Theta\), the equality set \(\{\Theta=\pm w\}\) has \(\mu\)-measure zero.  Hence \(\langle w,\Theta\rangle^2<1\) on a set of full \(\mu\)-measure, and the inequality in \eqref{eq:lmaxM_strict} is strict for every unit \(w\).  By continuity of \(w\mapsto w^{\mathsf T}Mw\) and compactness of the unit sphere in \(T_N\mathbb S^m\), the supremum over \(\|w\|=1\) is also strictly smaller than \(1\).  Thus \(\lambda_{\max}(M)<1\).
\end{proof}

\begin{theorem}[\textbf{Type I finite-sample smeariness in the non-smeary regime}]
\label{thm:typeI_hemisphere_general}
Let \(m\ge2\), and let \(\mu\) be absolutely continuous on \(\mathbb S^m\) with unique Fr\'echet mean \(N\).  Assume that \(\mu\) satisfies the setup of Section~\ref{scn:set-up-general-density} and Assumption~\ref{assum:density-assumption}.  Let \(\widehat\mu_n\) be a measurable sequence of intrinsic sample Fr\'echet means, set \(U_n:=\exp_N^{-1}(\widehat\mu_n)\), and let
\[
    \Sigma:=\operatorname{Cov}\bigl[\exp_N^{-1}(X)\bigr],
    \qquad
    H:=D^2\widetilde F_\mu(0).
\]
Assume that \(H\) is positive definite and that the non-smeary Fr\'echet mean CLT holds:
\[
    \sqrt n\,U_n\rightarrow_{\mathcal D}\mathcal N(0,4H^{-1}\Sigma H^{-1}) .
\]
Assume also Assumption~\ref{assump:momnt-bound}.  Then \(\mu\) is Type I finite sample smeary.  More precisely, the asymptotic modulation exists and satisfies \(m_\infty>1\).
\end{theorem}

\begin{proof}
By the CLT in the statement, Assumption~\ref{assump:momnt-bound}, and Lemma~\ref{lem:L2_from_CLT},
\[
    n\,\mathbb E\|U_n\|^2
    \longrightarrow
    \operatorname{trace}\bigl(4H^{-1}\Sigma H^{-1}\bigr)
    =
    4\,\operatorname{trace}(\Sigma H^{-2}).
\]
Since \(H\) is positive definite, \(\lambda_{\max}(H)I_m\succeq H\), and therefore \(H^{-2}\succeq \lambda_{\max}(H)^{-2}I_m\).  As \(\Sigma\succeq0\), this gives
\[
    \operatorname{trace}(\Sigma H^{-2})
    \ge
    \frac{1}{\lambda_{\max}(H)^2}\operatorname{trace}(\Sigma).
\]
By Lemma~\ref{lem:spectral_bound_Hess}, applied to \(M=\frac12H\), we have \(\lambda_{\max}(H)<2\).  Hence the asymptotic modulation from Definition~\ref{def:asymptotic_modulation} satisfies
\[
    m_\infty
    =
    \frac{\operatorname{trace}\bigl(4H^{-1}\Sigma H^{-1}\bigr)}
    {\operatorname{trace}(\Sigma)}
    \ge
    \frac{4}{\lambda_{\max}(H)^2}>1.
\]
Thus the Fr\'echet mean is Type I finite sample smeary.
\end{proof}

\begin{remark}[\textbf{Special case of hemispherical support}]\label{rmk:hemi-implies-FSS}
    Hemispherical support provides a natural non-smeary regime on spheres, so that the previous theorem can be applied to deduce Type-I FSS; see \citet{Pal2026SharpSupportThresholds} for the corresponding support-threshold picture. 
\end{remark}

\begin{remark}[\textbf{Previous Type-I FSS results}]
\label{rem:convex_closure_density}
  The theorem generalizes Theorem 6 of \citet{EltznerFSSonspheres} in the case of densities: no rotational symmetry has been assumed here.
\end{remark}

\begin{remark}[\textbf{On the necessity of the moment condition}]
\label{rem:moment_condition_importance}
The moment bound in Assumption~\ref{assump:momnt-bound} is not merely technical.  By Lemma~\ref{lem:L2_from_CLT}, it ensures uniform integrability of \(\{\|\sqrt n\,U_n\|^2\}_{n\ge1}\), which is what allows one to upgrade the non-smeary CLT to
\[
    n\,\mathbb E\|U_n\|^2
    \longrightarrow
    \operatorname{trace}\bigl(4H^{-1}\Sigma H^{-1}\bigr).
\]
Convergence in distribution alone does not imply convergence of second moments. 
\end{remark}

The next remark shows that, even in the non-smeary rotationally symmetric regime, the limiting modulation can be made arbitrarily large, similar to the construction with a measure carrying a point mass in \citet[Theorem~4.3]{Eltzner2022GeometricalSmeariness}.

\begin{remark}[\textbf{Arbitrarily high modulation under rotational symmetry}]
\label{rem:arbitrary_modulation_fixed_m}
Fix \(m\ge2\), and let \(R_m\in(\pi/2,\pi)\) denote the unique zero of \(b_m(R)=1+(m-1)R\cot R\).  Let \(\varepsilon>0\).  Then, at the level of the rotationally symmetric Hessian formula, one can choose a smooth rotationally symmetric density \(\rho(\|\cdot\|)\) on \(\mathbb S^m\), with radial support contained in a small interval \([\widetilde R_m-\eta,\widetilde R_m+\eta]\subset(\pi/2,R_m)\), for some \(\widetilde R_m<R_m\) sufficiently close to \(R_m\), such that \(H=D^2\widetilde F_\mu(0)\) is positive definite and satisfies \(\lambda_{\max}(H)<\varepsilon\).

Indeed, by \eqref{eqn:Hess}, under rotational symmetry one has
\[
    H
    =
    \frac{2\operatorname{vol}(\mathbb S^{m-1})}{m}
    \left(
    \int_0^\pi
    b_m(R)\rho(R)(\sin R)^{m-1}\,dR
    \right)I_m.
\]
Since \(b_m(R)>0\) on \((0,R_m)\) and \(b_m(R)\to0\) as \(R\uparrow R_m\), choosing \(\rho\) supported in a sufficiently small neighborhood of \(\widetilde R_m<R_m\), with \(\widetilde R_m\uparrow R_m\), yields \(0<\lambda_{\max}(H)<\varepsilon\).  Consequently, in the rotationally symmetric scalar-Hessian case, whenever the corresponding limiting modulation is defined, \(m_\infty=4/\lambda_{\max}(H)^2\) can be made arbitrarily large.  
\end{remark}


\section{Curse of dimensionality for modulation of non-smeary random variables}
\label{scn:curse-dim-modulation}

This section shows that, for dimensionally comparable rotationally symmetric densities supported in a closed sub-hemispherical cap around the unique Fr\'echet mean, the asymptotic modulation increases with dimension.  The effect is purely second-order: under rotational symmetry, the Hessian is a scalar multiple of the identity, and the scalar Hessian eigenvalue decreases with dimension.

\begin{definition}[\textbf{Dimensionally comparable rotationally symmetric family}]
\label{def:dimensionally-comparable-distributions}
Let \(N\) denote the north pole.  Fix \(\varepsilon\in[0,\frac{\pi}{2})\) and set \(R_\varepsilon:=\frac{\pi}{2}-\varepsilon\).  For each \(m\ge2\), let \(\mu^{(m)}\) be a probability measure on \(\mathbb S^m\), rotationally symmetric about \(N\), with
\[
\operatorname{supp}\mu^{(m)}\subset \overline B(N;R_\varepsilon).
\]
We say that \(\{\mu^{(m)}\}_{m\ge2}\) is a \emph{dimensionally comparable rotationally symmetric family} if there exists a measurable function \(w:[0,R_\varepsilon]\to[0,\infty)\) such that, for every \(m\ge2\),
\[
0<Z_m:=\int_{0}^{R_\varepsilon} w(R)(\sin R)^{m-1}\,dR<\infty,
\]
and, in normal polar coordinates \(v=R\Theta\in T_N\mathbb S^m\), the density of \(\mu^{(m)}\) with respect to polar volume measure has the form
\[
\rho^{(m)}(R,\Theta)=\rho_m(R)\,\varphi_m(\Theta),
\qquad
\rho_m(R):=\frac{w(R)}{Z_m},
\]
where \(\varphi_m\) is the uniform probability density on \(\mathbb S^{m-1}\).  Equivalently, the induced radial probability measure on \([0,R_\varepsilon]\) is
\[
\nu_m(dR)=\frac{w(R)(\sin R)^{m-1}}{\int_{0}^{R_\varepsilon} w(U)(\sin U)^{m-1}\,dU}\,dR.
\]
Thus the family is generated by a single dimension-independent unnormalized radial weight \(w\), while the dependence on \(m\) enters only through the Jacobian factor \((\sin R)^{m-1}\) and the normalizing constant \(Z_m\).
\end{definition}

\begin{expl}[\textbf{Dimensionally comparable distributions}]
The following are dimensionally comparable rotationally symmetric families.

\begin{enumerate}
\item \textbf{Uniform law on the truncated geodesic ball.}
Take \(w(R)\equiv 1\).  Then
\[
\nu_m(dR)=\frac{(\sin R)^{m-1}}{\int_{0}^{R_\varepsilon}(\sin U)^{m-1}\,dU}\,dR.
\]

\item \textbf{Truncated von Mises--Fisher family with fixed concentration parameter \(\kappa\).}
Since the von Mises--Fisher density is proportional to \(e^{\kappa N^\top x}\) and \(N^\top x=\cos R\), take \(w(R)=e^{\kappa\cos R}\).  Then
\[
\nu_m(dR)=\frac{e^{\kappa\cos R}(\sin R)^{m-1}}{\int_{0}^{R_\varepsilon} e^{\kappa\cos U}(\sin U)^{m-1}\,dU}\,dR.
\]
See \citet{JuppRegoliAzzalini2016}.  By rotational symmetry, \(N\) is a critical point of the Fr\'echet function.

\item \textbf{Truncated Watson family with fixed concentration parameter \(\kappa\).}
Since the Watson density is proportional to \(e^{\kappa(N^\top x)^2}\) and \((N^\top x)^2=\cos^2 R\), take \(w(R)=e^{\kappa\cos^2 R}\).  Then
\[
\nu_m(dR)=\frac{e^{\kappa\cos^2 R}(\sin R)^{m-1}}{\int_{0}^{R_\varepsilon} e^{\kappa\cos^2 U}(\sin U)^{m-1}\,dU}\,dR.
\]
See \citet{Watson1965,Mardia1975}.

\item \textbf{General truncated zonal family with fixed profile.}
More generally, if \(g:[-1,1]\to[0,\infty)\) is measurable, then \(w(R)=g(\cos R)\) defines a dimensionally comparable rotationally symmetric family.  See \citet{JuppRegoliAzzalini2016}.
\end{enumerate}
\end{expl}

The following theorem proves that, for the families in Definition~\ref{def:dimensionally-comparable-distributions}, the asymptotic modulation is strictly increasing with dimension.

\begin{theorem}[\textbf{Dimension monotonicity of asymptotic modulation for dimensionally comparable rotationally symmetric families}]
\label{thm:dim_monotone_modulation_rot}
Fix \(\varepsilon\in[0,\frac{\pi}{2})\) and set \(R_\varepsilon:=\frac{\pi}{2}-\varepsilon\).  For each integer \(m\ge2\), let \(\mu^{(m)}\) be an absolutely continuous probability measure on \(\mathbb S^m\) with unique intrinsic Fr\'echet mean \(N\), and assume that \(\{\mu^{(m)}\}_{m\ge2}\) is a dimensionally comparable rotationally symmetric family about \(N\), supported in \(\overline B(N;R_\varepsilon)\), in the sense of Definition~\ref{def:dimensionally-comparable-distributions}.  Assume further that the moment condition \ref{assump:momnt-bound} holds for each \(m\).  Let \(X^{(m)}\sim\mu^{(m)}\), set \(\Sigma_m:=\operatorname{Cov}[\exp_N^{-1}(X^{(m)})]\), and let \(H_m:=D^2\widetilde F_m(0)\).  Define the dimension-indexed asymptotic modulations, cf. Definition~\ref{def:asymptotic_modulation}, by
\[
m_\infty^{(m)}
:=
\frac{\operatorname{trace}\!\bigl(4H_m^{-1}\Sigma_m H_m^{-1}\bigr)}{\operatorname{trace}(\Sigma_m)}.
\]
Then:
\begin{enumerate}
\item The sequence \(\bigl(m_\infty^{(m)}\bigr)_{m\ge2}\) is strictly increasing with dimension:
\[
m_\infty^{(m+1)} > m_\infty^{(m)}
\qquad (m\ge2).
\]

\item If the common radial weight \(w\) from Definition~\ref{def:dimensionally-comparable-distributions} satisfies, for every \(\delta>0\),
\[
\int_{R_\varepsilon-\delta}^{R_\varepsilon} w(R)\,dR > 0,
\]
then, in the extended real line,
\[
\lim_{m\to\infty} m_\infty^{(m)}
=
\frac{1}{\bigl(R_\varepsilon \cot R_\varepsilon\bigr)^2}
=
\frac{1}{\Bigl(\left(\frac{\pi}{2}-\varepsilon\right)\cot\!\left(\frac{\pi}{2}-\varepsilon\right)\Bigr)^2}.
\]
\end{enumerate}
\end{theorem}

\begin{proof}
\textbf{Proof of (1).}
By Definition~\ref{def:dimensionally-comparable-distributions}, there exists a measurable function \(w:[0,R_\varepsilon]\to[0,\infty)\) such that \(Z_m:=\int_0^{R_\varepsilon}w(R)(\sin R)^{m-1}\,dR\in(0,\infty)\), \(\rho_m(R)=w(R)/Z_m\), and \(\rho^{(m)}(R,\Theta)=\rho_m(R)\varphi_m(\Theta)\), where \(\varphi_m\) is the uniform probability density on \(\mathbb S^{m-1}\).  Rotational symmetry gives
\[
\int_{\mathbb S^{m-1}}\Theta\Theta^\top\,\varphi_m(\Theta)\,d\Theta=\frac{1}{m}I_m.
\]
Using the Hessian formula \eqref{eqn:Hessian-general-density} with this normalized angular density, equivalently the rotationally symmetric reduction \eqref{eqn:Hess} with the angular volume factor cancelled by \(\varphi_m\), we obtain
\[
H_m
=
\frac{2}{m}\left(\int_0^{R_\varepsilon}\rho_m(R)(\sin R)^{m-1}b_m(R)\,dR\right)I_m.
\]
Since \(b_m(R)=1+(m-1)R\cot R\), this becomes \(H_m=\lambda_m I_m\), where
\[
s_m:=
\int_0^{R_\varepsilon}(R\cot R)\,\rho_m(R)(\sin R)^{m-1}\,dR
=
\frac{\int_0^{R_\varepsilon}(R\cot R)\,w(R)(\sin R)^{m-1}\,dR}{Z_m},
\qquad
\lambda_m:=2\left(s_m+\frac{1-s_m}{m}\right).
\]
Since \(0\le R\cot R<1\) on \((0,\pi/2]\) and \(Z_m\in(0,\infty)\), we have \(0\le s_m<1\).  Hence \(\lambda_m>0\), so \(H_m\) is positive definite.  The trace formula in Definition~\ref{def:asymptotic_modulation} therefore gives
\[
m_\infty^{(m)}
=
\frac{\operatorname{trace}\!\bigl((4/\lambda_m^2)\Sigma_m\bigr)}{\operatorname{trace}(\Sigma_m)}
=
\frac{4}{\lambda_m^2}.
\]
Thus \(m_\infty^{(m+1)} > m_\infty^{(m)}\) is equivalent to \(\lambda_{m+1}<\lambda_m\).

We next prove that \(s_{m+1}<s_m\).  Define the probability measure \(\nu_m\) on \([0,R_\varepsilon]\) by
\[
\nu_m(dR):=\frac{w(R)(\sin R)^{m-1}}{Z_m}\,dR.
\]
Let \(c_m:=\int_0^{R_\varepsilon}\sin R\,\nu_m(dR)=Z_{m+1}/Z_m>0\).  Then
\[
\frac{\sin R}{c_m}\,\nu_m(dR)=\nu_{m+1}(dR).
\]
Therefore
\[
s_{m+1}
=
\frac{1}{c_m}\int_0^{R_\varepsilon}(R\cot R)\sin R\,\nu_m(dR).
\]
Subtracting \(s_m=\int_0^{R_\varepsilon}(R\cot R)\nu_m(dR)\), and using the standard covariance symmetrization identity, gives
\[
s_{m+1}-s_m
=
\frac{1}{2c_m}\int_0^{R_\varepsilon}\int_0^{R_\varepsilon}
\bigl(R\cot R-U\cot U\bigr)(\sin R-\sin U)\,\nu_m(dR)\,\nu_m(dU).
\]
The map \(R\mapsto R\cot R\) is strictly decreasing on \((0,\pi/2]\), while \(R\mapsto\sin R\) is strictly increasing on \([0,\pi/2]\).  Hence the integrand is everywhere nonpositive.  Since \(\nu_m\) is absolutely continuous with respect to Lebesgue measure and is not supported on a single point, the inequality is strict.  Thus \(s_{m+1}<s_m\).

Finally,
\[
\lambda_{m+1}<\lambda_m
\;\Longleftrightarrow\;
m^2(s_{m+1}-s_m)<1-s_m.
\]
Since \(s_{m+1}-s_m<0\) and \(s_m<1\), this inequality holds.  Therefore \(\lambda_{m+1}<\lambda_m\), and hence \(m_\infty^{(m+1)}>m_\infty^{(m)}\).

\textbf{Proof of (2).}
Assume that for every \(\delta>0\), \(\int_{R_\varepsilon-\delta}^{R_\varepsilon}w(R)\,dR>0\).  Fix \(\delta\in(0,R_\varepsilon)\) and write \(q_\delta:=\sin(R_\varepsilon-\delta)\) and \(\beta_\delta:=\int_{R_\varepsilon-\delta/2}^{R_\varepsilon}w(R)\,dR>0\).  Then
\[
\int_0^{R_\varepsilon-\delta}w(R)(\sin R)^{m-1}\,dR
\le
q_\delta^{\,m-2}\int_0^{R_\varepsilon-\delta}w(R)\sin R\,dR,
\]
while
\[
Z_m\ge \beta_\delta\,\sin(R_\varepsilon-\delta/2)^{m-1}.
\]
Since \(q_\delta<\sin(R_\varepsilon-\delta/2)\), it follows that
\[
\nu_m([0,R_\varepsilon-\delta])\to0,
\qquad
\nu_m([R_\varepsilon-\delta,R_\varepsilon])\to1.
\]
As \(R\mapsto R\cot R\) is continuous on \([0,R_\varepsilon]\), we obtain
\[
s_m=\int_0^{R_\varepsilon}(R\cot R)\,\nu_m(dR)
\longrightarrow
R_\varepsilon\cot R_\varepsilon.
\]
Therefore \(\lambda_m=2(s_m+(1-s_m)/m)\to 2R_\varepsilon\cot R_\varepsilon\), and consequently
\[
\lim_{m\to\infty}m_\infty^{(m)}
=
\lim_{m\to\infty}\frac{4}{\lambda_m^2}
=
\frac{1}{\bigl(R_\varepsilon\cot R_\varepsilon\bigr)^2}.
\]
\end{proof}

\subsection{Numerical simulation of curse of dimensionality for modulation}
\label{subsec:numerical-curse-dim-modulation}

We carried out a numerical experiment for the uniform distribution supported on a sub-hemispherical cap in \(\mathbb S^m\).  More precisely, for each \(m\in\{2,3,5,10,50,100,200\}\), we considered the uniform distribution on the geodesic cap \(B(N;R_\varepsilon)\), with \(R_\varepsilon=\frac{\pi}{2}-10^{-4}\), centered at the north pole \(N\).  For each sample size \(n\in\{1000,2000,3000\}\), we generated independent samples \(X_1,\dots,X_n\) from this cap-uniform distribution, computed the sample intrinsic Fr\'echet mean \(\hat\mu_n\) by Riemannian gradient descent on \(\mathbb S^m\), and then evaluated the finite-sample quantity
\[
Z_n^{(m)}:=\frac{n\,d^2(\hat\mu_n,N)}{V_m},
\qquad
V_m:=\mathbb E\bigl[d^2(X,N)\bigr].
\]
For each pair \((n,m)\), the experiment was repeated three times and we recorded the empirical mean of \(Z_n^{(m)}\), shown in Figure~\ref{fig:finite-sample-modulation-vs-dimension}.

\begin{figure}[t]
\centering
\includegraphics[width=0.72\textwidth]{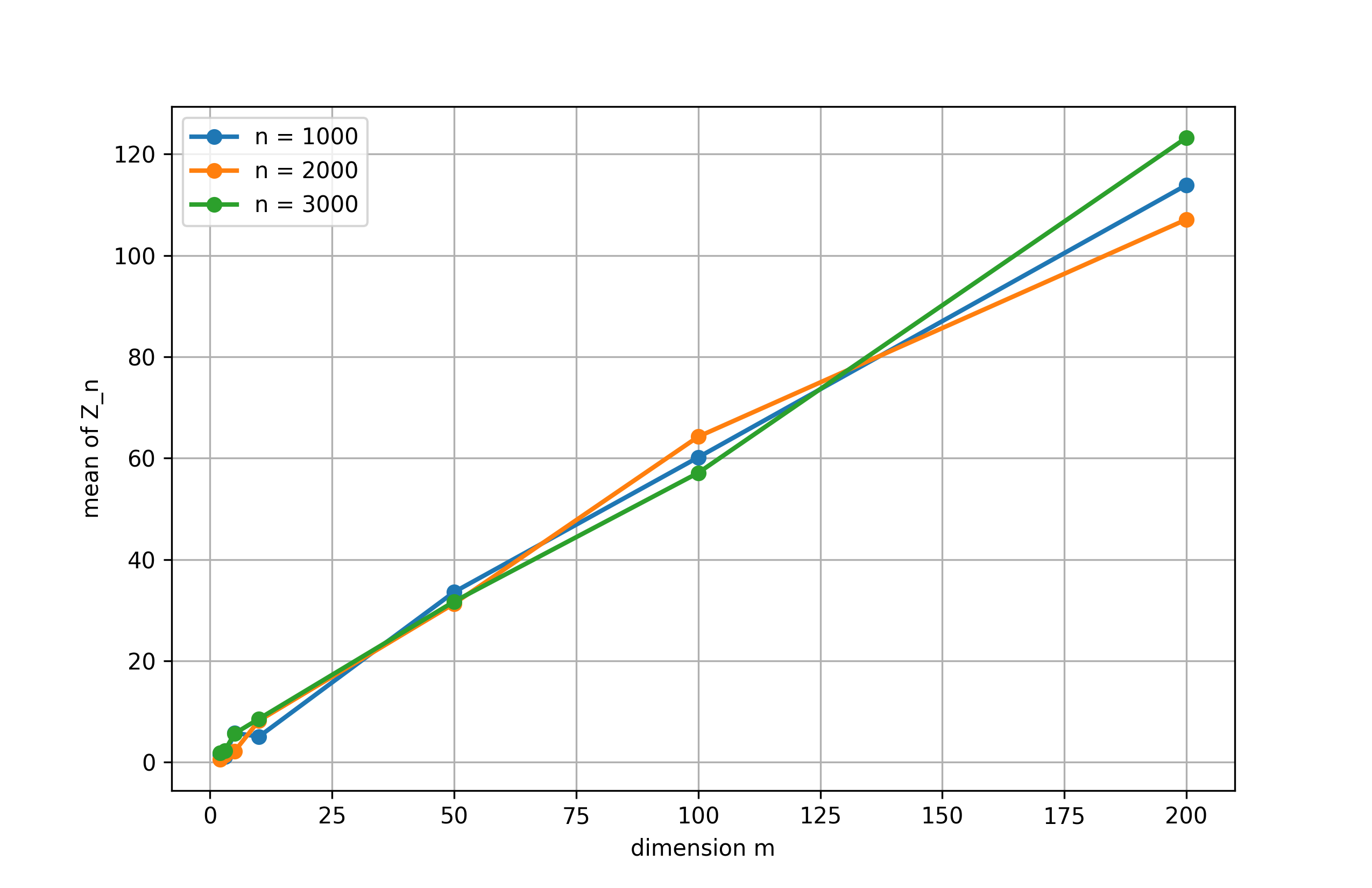}
\caption{\textbf{Finite-sample modulation versus dimension.}}
\label{fig:finite-sample-modulation-vs-dimension}
\end{figure}

\section{Application to Geometric Statistics: Hessian Estimation and modulation-aware tests}\label{scn:modulation-aware-tests}

Throughout this section, assume that the random variable $X$, or equivalently its induced measure $\mu$, has a unique population Fr\'echet mean $\bar\mu$. We shall use one of the following two support assumptions. In the general case, assume that there exists $R_\ast<\pi/2$ such that
\begin{assumption}\label{assum:bounded-supp-general}
      $\operatorname{supp}(X)\subset \overline{B}(\bar{\mu};R_\ast)$.
\end{assumption}
In the rotationally symmetric case, we allow the weaker condition that there exists $R_\ast<R_m$, where $R_m$ is the first zero of $b_m$ (cf. Equation~\ref{eqn:b}), such that
\begin{assumption}\label{assum:bounded-supp-rot}
      $\operatorname{supp}(X)\subset \overline{B}(\bar{\mu};R_\ast)$
      and $\mu$ is \textit{rotationally symmetric} about $\bar\mu$.
\end{assumption}

\nd Under Assumption~\ref{assum:bounded-supp-general}, the Hessian of the Fr\'echet function at $\bar\mu$ is positive definite by Theorem 6.1 of \citet{Pal2026SharpSupportThresholds}. Under Assumption~\ref{assum:bounded-supp-rot}, this also follows from Theorem 5.1 of the same paper. Let \(H=D^{2}\widetilde F(0)\) denote the Hessian of the \textit{lifted} Fr\'echet function in normal coordinates at $\bar\mu$.

Recall from Section~\ref{scn:set-up-general-density} that, when $\mu$ has density $\rho(R,\Theta)$ in the polar coordinates around $\bar\mu$, the general Hessian formula is
\[
H
=
2\int_{0}^{R_\ast}\int_{\mathbb S^{m-1}}
\left[
\frac{R}{\tan R}I_m+
\left(1-\frac{R}{\tan R}\right)\Theta\Theta^T
\right]
\rho(R,\Theta)(\sin R)^{m-1}\,d\Theta\,dR.
\]
In the rotationally symmetric case this reduces to (cf. Equation~\eqref{eqn:Hess})
\[
H
=
\frac{2 \operatorname{vol}(\mathbb{S}^{m-1})}{m} 
\left(
\int_0^{R_\ast}\rho(R)(\sin R)^{m-1}b_m(R)\,dR
\right)I_m.
\]

\subsection{Plug-in Hessian estimators in the general and rotationally symmetric cases}

\noindent
Let $X_1,\dots,X_n$ be i.i.d.\ copies of $X$, and let $\hat\mu_n$ be a measurable selection of sample Fréchet means. Define the \textit{plug-in tangent vectors} and \textit{radii}
\[
V_{i,n}:=\log_{\hat\mu_n}(X_i),
\qquad
R_{i,n}:=\|V_{i,n}\|=d(\hat\mu_n,X_i).
\]
For $R_{i,n}>0$, set $\Theta_{i,n}:=V_{i,n}/R_{i,n}$, with all tangent vectors represented in a smooth local orthonormal frame. In the general case, the plug-in Hessian estimator is obtained by replacing the population integral in the general Hessian formula, Equation~\eqref{eqn:Hessian-general-density}, by the empirical average:
\begin{equation}\label{eqn:Hess-est-gen}
\boxed{    H_n^{\mathrm{gen}}
:=
\frac{2}{n}\sum_{i=1}^n
\left[
\frac{R_{i,n}}{\tan R_{i,n}}I_m+
\left(1-\frac{R_{i,n}}{\tan R_{i,n}}\right)
\Theta_{i,n}\Theta_{i,n}^{T}
\right],}
\end{equation}

with the \textit{convention} that the $i$th summand is $I_m$ when $R_{i,n}=0$.

In the rotationally symmetric case, one may instead use the simpler radial plug-in estimator, obtained from Equation~\eqref{eqn:Hess},
\begin{equation}\label{eqn:Hess-est-rot-symm}
\boxed{    H_n^{\mathrm{rot}}
:=
\frac{2}{m}
\left(
\frac1n\sum_{i=1}^n b_m(R_{i,n})
\right)I_m.}
\end{equation}

Below, $H_n$ denotes $H_n^{\mathrm{gen}}$ (Equation \eqref{eqn:Hess-est-gen}) in the general case and $H_n^{\mathrm{rot}}$ in the rotationally symmetric case (Equation \eqref{eqn:Hess-est-rot-symm}) .

In the remainder of this section, \(H_n\) denotes \(H_n^{\mathrm{gen}}\) under Assumption~\ref{assum:bounded-supp-general}, and \(H_n^{\mathrm{rot}}\) under Assumption~\ref{assum:bounded-supp-rot}. The next result shows that $H_n$ is indeed a consistent estimator of $H,$ the Hessian of the Fréchet function.

\begin{proposition}[\textbf{Consistency of the plug-in Hessian estimators}]
\label{prop:Hn_consistency}
Under the corresponding support assumption,
\[
H_n\to H \quad \text{in probability}.
\]
\end{proposition}

\begin{proof}
We prove the general case, where \(H_n=H_n^{\mathrm{gen}}\). The rotationally symmetric case follows from the same argument applied to the scalar radial formula involving \(b_m\).

By the consistency results in \citet{Ziezold1977} for Fréchet sample mean sets, together with uniqueness of the population Fréchet mean \(\bar\mu\), any measurable selection \(\hat\mu_n\) of sample Fréchet means satisfies \(\hat\mu_n\to\bar\mu\) in probability.

Choose \(\varepsilon>0\) such that \(R_\ast+\varepsilon<\pi/2\), and set
\(
A_{n,\varepsilon}:=\{d(\hat\mu_n,\bar\mu)\le \varepsilon\}.
\)
On \(A_{n,\varepsilon}\), for every \(x\in\operatorname{supp}(X)\),
\[
d(\hat\mu_n,x)\le d(\hat\mu_n,\bar\mu)+d(\bar\mu,x)
\le \varepsilon+R_\ast<\pi/2.
\]
Hence the logarithm map and the pointwise Hessian of squared distance are smooth on the compact set
\(
\overline B(\bar\mu;\varepsilon)\times \operatorname{supp}(X).
\)

Let \(\Psi(p,x)\) denote, in the chosen smooth local orthonormal frame, the pointwise Hessian contribution
\[
\Psi(p,x)
=
2\left[
\frac{r}{\tan r}I_m+
\left(1-\frac{r}{\tan r}\right)\theta\theta^T
\right],
\qquad
r=d(p,x),\quad \theta=\frac{\log_p(x)}{r},
\]
with the convention \(\Psi(p,p)=2I_m\). By smoothness on the compact set above, there exists \(C<\infty\) such that, on \(A_{n,\varepsilon}\),
\[
\|\Psi(\hat\mu_n,X_i)-\Psi(\bar\mu,X_i)\|
\le C\,d(\hat\mu_n,\bar\mu)
\]
for every \(i=1,\dots,n\). Define
\[
\widetilde H_n:=\frac1n\sum_{i=1}^n \Psi(\bar\mu,X_i).
\]
By the weak law of large numbers applied entrywise,
\[
\widetilde H_n\to \int_{\mathbb S^m}\Psi(\bar\mu,x)\,d\mu(x)=H
\qquad\text{in probability}.
\]
Moreover, on \(A_{n,\varepsilon}\),
\[
\|H_n-\widetilde H_n\|
=
\left\|
\frac1n\sum_{i=1}^n
\bigl(\Psi(\hat\mu_n,X_i)-\Psi(\bar\mu,X_i)\bigr)
\right\|
\le C\,d(\hat\mu_n,\bar\mu).
\]
Since \(\mathbb P(A_{n,\varepsilon})\to1\) and \(d(\hat\mu_n,\bar\mu)\to0\) in probability, it follows that \(H_n-\widetilde H_n\to0\) in probability. Combining this with \(\widetilde H_n\to H\) proves \(H_n\to H\) in probability.

For \(H_n^{\mathrm{rot}}\), the same proof applies with the matrix-valued map \(\Psi\) replaced by the scalar Lipschitz function \(r\mapsto b_m(r)\) on \([0,R_\ast+\varepsilon]\subset[0,R_m)\). This proves the rotationally symmetric case as well.
\end{proof}

\subsection{A Hessian-corrected quadratic form}

\noindent
Let
\[
\bar U_n:=\log_{\bar{\mu}}(\hat{\mu}_n),
\qquad
\Sigma:=\operatorname{Cov}(\log_{\bar{\mu}}(X)),
\qquad
\Gamma:=4H^{-1}\Sigma H^{-1}.
\]
Let $\Sigma_n$ be a symmetric estimator of $\Sigma$ such that $\Sigma_n\to\Sigma$ in probability.

\begin{proposition}[\textbf{Quadratic form limit}]
\label{prop:quadratic_chisq}
Assume that $H$ is positive definite, and assume that the non-smeary central limit theorem holds:
\[
\sqrt n\,\bar U_n \rightarrow_{\mathcal D}\mathcal N(0,\Gamma).
\]
Assume moreover that $\Sigma$ is positive definite and that $\Sigma_n\to\Sigma$ in probability. Define
\[
E_n:=\{H_n \text{ is invertible and } \Sigma_n \text{ is positive definite}\}.
\]
Set
\begin{equation}\label{eqn:Gamma-n,T-n}
\widehat\Gamma_n:=
\begin{cases}
4H_n^{-1}\Sigma_n H_n^{-1}, & \text{on }E_n,\\
I_m, & \text{on }E_n^c,
\end{cases}
\qquad
T_n:=
\begin{cases}
n\,\bar U_n^{T}\widehat\Gamma_n^{-1}\bar U_n, & \text{on }E_n,\\
0, & \text{on }E_n^c.
\end{cases}
\end{equation}
Then
\[
T_n\rightarrow_{\mathcal D}\chi_m^2.
\]
\end{proposition}

\begin{proof}
By Proposition~\ref{prop:Hn_consistency}, $H_n\to H$ in probability. Since $H$ is positive definite and therefore invertible, so
\[
\mathbb P(H_n \text{ is invertible})\to 1.
\]
Similarly, since $\Sigma_n\to\Sigma$ in probability and $\Sigma$ is positive definite,
\(
\mathbb P(\Sigma_n \text{ is positive definite})\to 1.
\)
Hence
\(
\mathbb P(E_n)\to 1.
\)
Since
\[
H_n\to H
\quad\text{and}\quad
\Sigma_n\to\Sigma
\qquad\text{in probability},
\]
and since the map
\(
(A,B)\mapsto 4A^{-1}BA^{-1}
\)
is continuous at $(H,\Sigma)$ on the cone of positive definite matrices, it follows that
\(
\widehat\Gamma_n\to\Gamma
\quad\text{in probability}.
\)
Since $\Gamma$ is positive definite, continuity of inversion and matrix square root on the cone of positive definite matrices yields
\(
\widehat\Gamma_n^{-1/2}\to \Gamma^{-1/2}
\quad\text{in probability}.
\)
By Slutsky's theorem,
\(
\widehat\Gamma_n^{-1/2}\sqrt n\,\bar U_n
\rightarrow_{\mathcal D}
\Gamma^{-1/2}Z,
\, Z\sim\mathcal N(0,\Gamma).
\)
Since $\Gamma^{-1/2}Z\sim\mathcal N(0,I_m)$, the continuous mapping theorem applied to $\norm{.}^2$ gives
\(
\bigl\|\widehat\Gamma_n^{-1/2}\sqrt n\,\bar U_n\bigr\|^2
\rightarrow_{\mathcal D}\chi_m^2.
\)
Finally, for every $\delta>0$,
\[
\mathbb P\!\left(
\bigl|T_n-\|\widehat\Gamma_n^{-1/2}\sqrt n\,\bar U_n\|^2\bigr|>\delta
\right)
\le \mathbb P(E_n^c)\to 0,
\]
because on $E_n$ the two quantities coincide. Therefore
\[
T_n\rightarrow_{\mathcal D}\chi_m^2.
\]
\end{proof}

\medskip

\begin{theorem}[\textbf{Consistent estimator of tangential covariance $\Sigma$ above, and a modulation-aware one-sample test statistic}]
\label{thm:mod-aware-test}
Assume the corresponding support assumption stated at the beginning of this section, namely Assumption~\ref{assum:bounded-supp-general} in the general case or Assumption~\ref{assum:bounded-supp-rot} in the rotationally symmetric case.  Assume that \(H\) and \(\Sigma\) are positive definite, and that the non-smeary CLT
\(\sqrt n\,\bar U_n\rightarrow_{\mathcal D}\mathcal N(0,\Gamma),\) where  \(\Gamma:=4H^{-1}\Sigma H^{-1}\). Fix a smooth local orthonormal frame on a neighborhood of $\bar{\mu}$, and represent all tangent vectors in this frame. Define
\[
V_{i,n}:=\log_{\hat{\mu}_n}(X_i),
\qquad
\bar V_n:=\frac1n\sum_{i=1}^n V_{i,n},
\]
and set the tangential sample covariance matrix
\[
\Sigma_n:=\frac1n\sum_{i=1}^n (V_{i,n}-\bar V_n)(V_{i,n}-\bar V_n)^T.
\]
Then
\[
\Sigma_n\to \Sigma
\qquad\text{in probability}.
\]
Consequently, in Equation \eqref{eqn:Gamma-n,T-n},
we have
\[
\widehat\Gamma_n\to \Gamma
\qquad\text{in probability},
\qquad
T_n\rightarrow_{\mathcal D}\chi_m^2.
\]
\end{theorem}

\begin{proof}
Set
\[
V_i:=\log_{\bar{\mu}}(X_i),
\qquad
\widetilde{\bar V}_n:=\frac1n\sum_{i=1}^n V_i,
\qquad
\widetilde\Sigma_n:=\frac1n\sum_{i=1}^n (V_i-\widetilde{\bar V}_n)(V_i-\widetilde{\bar V}_n)^T.
\]
Since $\operatorname{supp}(X)\subset \overline B(\bar{\mu};R_\ast)$ with $R_\ast<R_m<\pi$, we may choose $\varepsilon>0$ such that
\(
R_\ast+\varepsilon<R_m <\pi.
\)
Then, for every $p\in \overline B(\bar{\mu};\varepsilon)$ and every $x\in \operatorname{supp}(X)$,
\(
d(p,x)\le d(p,\bar{\mu})+d(\bar{\mu},x)\le \varepsilon+R_\ast< R_m < \pi.
\)
Hence $x$ stays away from the cut locus of $p$, and therefore the map
\(
(p,x)\mapsto \log_p(x)
\)
is smooth on the compact set
\(
\overline B(\bar{\mu};\varepsilon)\times \operatorname{supp}(X).
\)
In particular, there exists $C<\infty$ such that on the event
\(
A_{n,\varepsilon}:=\{d(\hat{\mu}_n,\bar{\mu})\le \varepsilon\},
\),
one has
\[
\|V_{i,n}-V_i\|
=
\|\log_{\hat{\mu}_n}(X_i)-\log_{\bar{\mu}}(X_i)\|
\le C\,d(\hat{\mu}_n,\bar{\mu})
\]
for every $i=1,\dots,n$.

Since $\hat{\mu}_n\to\bar{\mu}$ in probability by the consistency theorem in \citet{Ziezold1977}, and uniqueness of the population Fréchet mean, we have
\(
\mathbb P(A_{n,\varepsilon})\to 1.
\)

Using the definitions
\[
\bar V_n=\frac1n\sum_{i=1}^n V_{i,n},
\qquad
\widetilde{\bar V}_n=\frac1n\sum_{i=1}^n V_i,
\]
we obtain, on $A_{n,\varepsilon}$,
\[
\|\bar V_n-\widetilde{\bar V}_n\|
=
\left\|\frac1n\sum_{i=1}^n (V_{i,n}-V_i)\right\|
\le
\frac1n\sum_{i=1}^n \|V_{i,n}-V_i\|
\le
\frac1n\sum_{i=1}^n C\,d(\hat{\mu}_n,\bar{\mu})
=
C\,d(\hat{\mu}_n,\bar{\mu}).
\]

\nd Next, since
\(
\|V_i\| = d(\mubar, X_i)\le R_\ast,
\|V_{i,n}\| = d(\hat{\mu}_n, X_i)\le R_\ast+\varepsilon
\)
on $A_{n,\varepsilon}$, using the Lipschitz continuity on \textit{bounded sets} of $v\mapsto vv^{\top}$, there exists $C'<\infty$ such that
\[
\|V_{i,n}V_{i,n}^T-V_iV_i^T\|
\le C'\,\|V_{i,n}-V_i\|
\le C'C\,d(\hat{\mu}_n,\bar{\mu})
\]
for every $i=1,\dots,n$. Therefore, on $A_{n,\varepsilon}$,
\[
\left\|
\frac1n\sum_{i=1}^n V_{i,n}V_{i,n}^T
-
\frac1n\sum_{i=1}^n V_iV_i^T
\right\|
\le C'C\,d(\hat{\mu}_n,\bar{\mu}).
\]

\nd Using the identity
\[
\Sigma_n=
\frac1n\sum_{i=1}^n V_{i,n}V_{i,n}^T-\bar V_n\bar V_n^T,
\qquad
\widetilde\Sigma_n=
\frac1n\sum_{i=1}^n V_iV_i^T-\widetilde{\bar V}_n\widetilde{\bar V}_n^T,
\]
it follows that
\[
\Sigma_n-\widetilde\Sigma_n\to 0
\qquad\text{in probability}.
\]
On the other hand, by the strong law of large numbers,
\(
\widetilde\Sigma_n\to \Sigma
\text{ a.s.}
\)
Hence
\(
\Sigma_n\to\Sigma
\text{ in probability}.
\)

\nd The convergence
\(
\widehat\Gamma_n\to\Gamma \text{ in Equation } \eqref{eqn:Gamma-n,T-n}
\text{ in probability}
\)
now follows exactly as in Proposition~\ref{prop:quadratic_chisq}, using Proposition~\ref{prop:Hn_consistency}. The conclusion
\(
T_n\rightarrow_{\mathcal D}\chi_m^2
\)
is then an immediate application of Proposition~\ref{prop:quadratic_chisq}, with the above practical choice of $\Sigma_n$.
\end{proof}

\noindent
\begin{remark}[\textbf{Relation with earlier FSS corrections and novelty of the present approach}]
The statistic $T_n$ in Theorem~\ref{thm:mod-aware-test}, (cf. Equation \eqref{eqn:Gamma-n,T-n}) provides a direct Hessian-corrected quadratic form for asymptotic inference in the non-smeary regime. This should be contrasted with the finite sample smeariness literature, where correction of quantile-based inference was carried out via bootstrap procedures; see, for example, Section~4 of \cite{EltznerFSSonspheres} and also \cite{HH2}. The reason is that in that line of work one does not build an \textit{explicit plug-in estimator} for the Hessian correction entering the covariance term $4H^{-1}\Sigma H^{-1}$, and so bootstrap-based calibration is used instead.

In the present sphere setting, by contrast, the Hessian is \textit{explicit} for measures with a density in both rotationally symmetric (Equation \eqref{eqn:Hess}) as well as general cases (Equation \eqref{eqn:Hessian-general-density}); in the rotationally symmetric case it reduces to the particularly simple scalar form used above. This makes it straightforward to construct a consistent plug-in estimator $H_n$ and hence a consistent estimator
\[
\widehat\Gamma_n=4H_n^{-1}\Sigma_n H_n^{-1}.
\]
Accordingly, the resulting quadratic form $T_n$ yields an explicit modulation-aware testing procedure based on a directly \textit{estimated Hessian correction}, rather than on\textit{ bootstrap calibration}. We now use this estimator $T_n$ for inference in the next Section \ref{sec:numerical-bootstrap-vs-hessian}.
\end{remark}

%

\subsection{Modulation-aware, Hessian-corrected, two-sample Hotelling statistic}

We now record the two-sample analogue of Theorem~\ref{thm:mod-aware-test}. This is the version used for testing equality of two Fréchet means.

\begin{theorem}[\textbf{Two-sample Hessian-corrected Hotelling statistic with possibly unequal covariance where pooled covariance is invertible (Behrens-Fisher problem)}]
\label{thm:two-sample-mod-aware-hotelling}
Let \(X_{1,1},\dots,X_{1,n_1}\) and \(X_{2,1},\dots,X_{2,n_2}\) be two independent samples on \(\mathbb S^m\), with laws \(\mu_1,\mu_2\) and unique population Fr\'echet means \(\bar\mu_1,\bar\mu_2\).  Consider the null hypothesis \(H_0:\bar\mu_1=\bar\mu_2=:\bar\mu\).  For \(j=1,2\), let \(F_j(p):=\int_{\mathbb S^m}d^2(p,x)\,d\mu_j(x)\), and define, under \(H_0\),
\[
H_j:=D^2(F_j\circ\exp_{\bar\mu})(0),
\qquad
\Sigma_j:=\operatorname{Cov}(\log_{\bar\mu}X_j),
\qquad
\Gamma_j:=4H_j^{-1}\Sigma_jH_j^{-1}.
\]
Assume that, group-wise under \(H_0\), the corresponding support condition from Assumption~\ref{assum:bounded-supp-general} or Assumption~\ref{assum:bounded-supp-rot} holds.  Assume also that \(H_j\) and \(\Sigma_j\) are positive definite, and that the non-smeary CLT holds for each group:
\[
\sqrt{n_j}\log_{\bar\mu}(\hat\mu_{j,n_j})
\rightarrow_{\mathcal D}
\mathcal N(0,\Gamma_j),
\qquad j=1,2,
\qquad n_j \to \infty
\]
where \(\hat\mu_{j,n_j}\) is a measurable sample Fr\'echet mean of the \(j\)-th sample.  Let \(n:=n_1+n_2\), and assume \(n_1/n\to\lambda\in(0,1)\), hence \(n_2/n\to1-\lambda\).

Define the group-wise sample Fréchet functions by \(F_{j,n_j}(p):=n_j^{-1}\sum_{i=1}^{n_j}d^2(p,X_{j,i})\), \(j=1,2\), and let \(\hat\mu_{0,n}\) be a measurable pooled sample Fr\'echet mean, that is,
\[
\hat\mu_{0,n}\in \operatorname*{arg\,min}_{p\in\mathbb S^m} F_{0,n}(p),
\qquad
F_{0,n}(p):=\frac{n_1}{n}F_{1,n_1}(p)+\frac{n_2}{n}F_{2,n_2}(p)
=\frac1n\sum_{j=1}^2\sum_{i=1}^{n_j}d^2(p,X_{j,i}).
\]

Represent all tangent vectors below in a smooth local orthonormal frame near \(\bar\mu\).  For \(j=1,2\) and \(i=1,\dots,n_j\), set \(V^0_{j,i}:=\log_{\hat\mu_{0,n}}(X_{j,i})\), \(R^0_{j,i}:=\|V^0_{j,i}\|\), and \(\Theta^0_{j,i}:=V^0_{j,i}/R^0_{j,i}\) when \(R^0_{j,i}>0\).  Define \(\widehat H^0_{j,n_j}\) by the group-wise version of the plug-in Hessian estimator given by \eqref{eqn:Hess-est-gen} in the general case, or by \eqref{eqn:Hess-est-rot-symm} in the rotationally symmetric case, using \(\hat\mu_{0,n}\) as base point.  Define
\[
\widehat\Sigma^0_{j,n_j}
:=
\frac1{n_j}\sum_{i=1}^{n_j}
(V^{0}_{j,i}-\bar V^0_j)(V^{0}_{j,i}-\bar V^0_j)^T,
\qquad
\bar V^0_j:=\frac1{n_j}\sum_{i=1}^{n_j}V^{0}_{j,i},
\qquad
j=1, 2,
\]
and, whenever the inverses exist, set
\[
\widehat\Gamma^0_{j,n_j}:=
4(\widehat H^0_{j,n_j})^{-1}
\widehat\Sigma^0_{j,n_j}
(\widehat H^0_{j,n_j})^{-1}.
\]
Finally define
\[
D_{n_1,n_2}:=
\log_{\hat\mu_{0,n}}(\hat\mu_{1,n_1})
-
\log_{\hat\mu_{0,n}}(\hat\mu_{2,n_2}).
\]
On the event
\[
E_n:=
\left\{
\widehat H^0_{1,n_1},\widehat H^0_{2,n_2}
\text{ are invertible and }
\frac{\widehat\Gamma^0_{1,n_1}}{n_1}
+
\frac{\widehat\Gamma^0_{2,n_2}}{n_2}
\text{ is invertible}
\right\},
\]
define
\begin{equation}
\label{eqn:two-sample-hotelling-stat}
T_{n_1,n_2}
:=
D_{n_1,n_2}^T
\left(
\frac{\widehat\Gamma^0_{1,n_1}}{n_1}
+
\frac{\widehat\Gamma^0_{2,n_2}}{n_2}
\right)^{-1}
D_{n_1,n_2}.
\end{equation}
On \(E_n^c\), set \(T_{n_1,n_2}:=0\).  Then, under \(H_0\),
\[
T_{n_1,n_2}\rightarrow_{\mathcal D}\chi^2_m,\, n_1, n_2 \to \infty, \, \frac{n_1}{n_1+n_2} \to \lambda \in (0, 1) .
\]
Thus the asymptotic level-\(\alpha\) test rejects \(H_0\) whenever \(T_{n_1,n_2}>\chi^2_{m,1-\alpha}\).
\end{theorem}

\begin{proof}
Under \(H_0\), the two population Fr\'echet means coincide at \(\bar\mu\).  We take the measurable selections \(\hat\mu_{1,n_1}\) and \(\hat\mu_{2,n_2}\) to be group-wise, that is, each depends only on its own sample.  Hence, by independence of the two samples, the two group-wise sample Fr\'echet means are independent.  Therefore the group-wise non-smeary CLTs, independence, and \(n_1/n\to\lambda\in(0,1)\) imply
\[
\sqrt n\bigl(\log_{\bar\mu}(\hat\mu_{1,n_1})
-
\log_{\bar\mu}(\hat\mu_{2,n_2})\bigr)
\rightarrow_{\mathcal D}
\mathcal N(0,\Gamma_\lambda),
\qquad
\Gamma_\lambda:=\frac{\Gamma_1}{\lambda}+\frac{\Gamma_2}{1-\lambda}.
\]
We first note that \(\hat\mu_{0,n}\to\bar\mu\) in probability.  Indeed, by definition,
\[
F_{0,n}=\frac{n_1}{n}F_{1,n_1}+\frac{n_2}{n}F_{2,n_2}.
\]
Since the sphere is compact, the group-wise empirical Fr\'echet functions \(F_{j,n_j}\) converge uniformly in probability to \(F_j\), \(j=1,2\).  Together with \(n_1/n\to\lambda\in(0,1)\), this gives uniform convergence of \(F_{0,n}\) to
\[
F_0:=\lambda F_1+(1-\lambda)F_2.
\]
Under \(H_0\), both \(F_1\) and \(F_2\) have the same unique minimizer \(\bar\mu\); hence the positive weighted sum \(F_0\) also has unique minimizer \(\bar\mu\).  The argmin consistency argument (see e.g.  \citet{vanderVaart1998}, Theorem 5.7, and apply this to $-F_{0,n}$) therefore gives \(\hat\mu_{0,n}\to\bar\mu\) in probability.

We next justify replacing the fixed base point \(\bar\mu\) by \(\hat\mu_{0,n}\).  By the support assumptions, all relevant points lie, with probability tending to one, in a common compact neighborhood on which the logarithm map is smooth away from the cut locus.  Work in a smooth local orthonormal frame near \(\bar\mu\), and write
\(a:=\log_{\bar\mu}q\), \(b_i:=\log_{\bar\mu}p_i\).  In this frame set
\[
L(a,b):=\log_{\exp_{\bar\mu}a}(\exp_{\bar\mu}b).
\]
The map \(L\) is smooth near \((0,0)\), and \(L(0,b)=b\).  Hence, by the fundamental theorem of calculus,
\[
\begin{aligned}
&L(a,b_1)-L(a,b_2)-\{L(0,b_1)-L(0,b_2)\} \\
&\qquad =
\int_0^1
\bigl(D_bL(a,b_2+t(b_1-b_2))-D_bL(0,b_2+t(b_1-b_2))\bigr)(b_1-b_2)\,dt .
\end{aligned}
\]
Since \(D_bL\) is smooth, the last display is
\(O(|a|\,|b_1-b_2|)\), hence, equivalently, for \(p_1,p_2,q\) sufficiently close to \(\bar\mu\),
\[
\begin{aligned}
&\log_q(p_1)-\log_q(p_2)
-
\bigl(\log_{\bar\mu}(p_1)-\log_{\bar\mu}(p_2)\bigr)
\\
&\qquad =
O\!\left(
d(q,\bar\mu)\{d(p_1,\bar\mu)+d(p_2,\bar\mu)\}
\right).
\end{aligned}
\]

Apply this with \(q=\hat\mu_{0,n}\) and \(p_j=\hat\mu_{j,n_j}\).  Since \(\hat\mu_{0,n}\to\bar\mu\) in probability and the group-wise CLTs imply \(d(\hat\mu_{j,n_j},\bar\mu)=O_{\mathbb P}(n_j^{-1/2})=O_{\mathbb P}(n^{-1/2})\) (this is immediate once we use the fact that a sequence of random variable converging in distribution to a random variable is tight), we obtain
\[
D_{n_1,n_2}
=
\log_{\bar\mu}(\hat\mu_{1,n_1})
-
\log_{\bar\mu}(\hat\mu_{2,n_2})
+
o_{\mathbb P}(n^{-1/2}).
\]
Consequently,
\[
\sqrt n\,D_{n_1,n_2}
\rightarrow_{\mathcal D}
\mathcal N(0,\Gamma_\lambda).
\]

The consistency proof of Proposition~\ref{prop:Hn_consistency} is local and uniform for base points in a sufficiently small neighborhood of \(\bar\mu\).  Since \(\hat\mu_{0,n}\to\bar\mu\) in probability, the same group-wise argument with the random base point \(\hat\mu_{0,n}\) gives
\[
\widehat H^0_{j,n_j}\to H_j
\qquad\text{in probability},\qquad j=1,2.
\]
Similarly, the argument of Theorem~\ref{thm:mod-aware-test}, again used locally uniformly in the base point, gives
\[
\widehat\Sigma^0_{j,n_j}\to\Sigma_j
\qquad\text{in probability},\qquad j=1,2.
\]
Hence \(\widehat\Gamma^0_{j,n_j}\to\Gamma_j\) in probability for \(j=1,2\).  Since \(n_1/n\to\lambda\) and \(n_2/n\to1-\lambda\),
\[
C_n:=
n\left(
\frac{\widehat\Gamma^0_{1,n_1}}{n_1}
+
\frac{\widehat\Gamma^0_{2,n_2}}{n_2}
\right)
\to
\Gamma_\lambda
\qquad\text{in probability}.
\]
The matrix \(\Gamma_\lambda\) is positive definite because \(\Gamma_1,\Gamma_2\) are positive definite and \(\lambda\in(0,1)\).  Therefore, by continuity of the smallest eigenvalue,
\[
\lambda_{\min}(C_n)\to \lambda_{\min}(\Gamma_\lambda)>0
\qquad\text{in probability}.
\]
Since multiplication by the positive scalar \(n\) does not affect invertibility, it follows that
\[
\mathbb P\!\left(
\frac{\widehat\Gamma^0_{1,n_1}}{n_1}
+
\frac{\widehat\Gamma^0_{2,n_2}}{n_2}
\text{ is invertible}
\right)\to1.
\]
Together with \(\widehat H^0_{j,n_j}\to H_j\) and \(H_j\) positive definite, which imply
\[
\mathbb P\bigl(\widehat H^0_{1,n_1},\widehat H^0_{2,n_2}
\text{ are invertible}\bigr)\to1,
\]
we obtain \(\mathbb P(E_n)\to1\).

On \(E_n\), the statistic can be written as
\[
T_{n_1,n_2}
=
(\sqrt n\,D_{n_1,n_2})^T
C_n^{-1}
(\sqrt n\,D_{n_1,n_2}).
\]
Since \(\sqrt n\,D_{n_1,n_2}\rightarrow_{\mathcal D}\mathcal N(0,\Gamma_\lambda)\) and \(C_n\to\Gamma_\lambda\) in probability, Slutsky's theorem and the continuous mapping theorem give
\[
T_{n_1,n_2}
\rightarrow_{\mathcal D}
Z^T\Gamma_\lambda^{-1}Z,
\qquad
Z\sim\mathcal N(0,\Gamma_\lambda).
\]
Thus \(Z^T\Gamma_\lambda^{-1}Z\sim\chi^2_m\).  Since \(\mathbb P(E_n^c)\to0\), the convention \(T_{n_1,n_2}=0\) on \(E_n^c\) does not affect the limiting distribution.  This proves the claim.
\end{proof}
\section{Numerical experiments on synthetic data: bootstrap-based versus Hessian-corrected tests}
\label{sec:numerical-bootstrap-vs-hessian}

\subsection{Experiments with synthetic dataset}

This section reports numerical experiments comparing three one-sample tests for the Fréchet mean on spheres, given the iid sample $\{X_1, X_2, \dots X_n\}$ that will be generated in a way to be made precise.

The first test is the \emph{naive quantile-based test}.  In the numerical experiments below, the null hypothesis is \(H_0:\mubar=N\).  We therefore measure the displacement of the sample Fr\'echet mean from the null mean \(N\), and set
\[
U_{0,n}:=\log_N(\hat\mu_n).
\]
The \textit{naive} quantile-based statistic is
\[
Q_n^{\mathrm{naive}}
:=
n\,U_{0,n}^{T}\widehat\Sigma_{0,n}^{-1}U_{0,n},
\]
where \(\widehat\Sigma_{0,n}\) is the empirical covariance matrix of the tangent vectors
\(\log_N(X_1),\dots,\log_N(X_n)\).  The test rejects for large values of
\(Q_n^{\mathrm{naive}}\), using the \(\chi_m^2\) quantile as critical value.  This is the direct Euclidean tangent-space calibration: it would be \(\chi_m^2\)-calibrated if the Hessian correction were absent, as in the Euclidean case, where $H=2I_m$.  In the non-smeary curved setting, however, the limiting covariance contains the Hessian correction \(\Gamma=4H^{-1}\Sigma H^{-1}\); see Section~\ref{scn:modulation-aware-tests}.

The second test is the \emph{bootstrap-based test}. It keeps the same testing problem, but replaces
the direct \(\chi_m^2\)-quantile calibration by bootstrap calibration. This is the type of correction
used in the finite sample smeariness literature to address the poor finite-sample behavior of the
naive quantile approximation; see Section~4 of \cite{EltznerFSSonspheres} and also \cite{HH2}.

The third test is the novel \emph{modulation-aware Hessian-corrected test} developed in
Section~\ref{scn:modulation-aware-tests}.  It uses the tangent vector from the hypothesized mean
\(N\) to the sample Fr\'echet mean,
\[
U_{0,n}:=\log_N(\hat\mu_n),
\]
and the statistic
\[
T_n
=
n\,U_{0,n}^{T}\widehat\Gamma_n^{-1}U_{0,n},
\qquad
\widehat\Gamma_n
=
4H_n^{-1}\Sigma_n H_n^{-1},
\]
where \(H_n\) is the plug-in Hessian estimator from Proposition~\ref{prop:Hn_consistency}, and
\(\Sigma_n\) is the empirical tangential covariance estimator from
Theorem~\ref{thm:mod-aware-test}.  In the rotationally symmetric simulations below, \(H_n\) is the
estimator in Equation~\eqref{eqn:Hess-est-rot-symm} in the rotationally symmetric regime.  By Proposition~\ref{prop:quadratic_chisq}
and Theorem~\ref{thm:mod-aware-test}, under \(H_0\) and in the non-smeary regime,
\[
T_n\rightarrow_{\mathcal D}\chi_m^2
\qquad\text{as } n\to\infty .
\]
Thus this third test also uses a \(\chi_m^2\) critical value, but with the estimated \textit{Hessian correction} (pre-multiplying by $4H_n^{-1}$ and post-multiplying by $H_n^{-1}$) included through 
\(\widehat\Gamma_n\).  The comparison below is therefore
between
\begin{enumerate}
    \item ignoring the Hessian correction,
    \item calibrating by bootstrap,
    \item  and taking into account the Hessian
correction directly.
\end{enumerate}

   The purpose is to illustrate the practical effect of the Hessian correction developed
in Section~\ref{scn:modulation-aware-tests}, especially in Theorem~\ref{thm:mod-aware-test}.
\subsection{Truncated von Mises--Fisher (tvMF) distribution as a hemisphrically supported, non-smeary distribution}\label{sec:tvMF}
We use the term truncated von Mises--Fisher distribution for the
von Mises--Fisher density restricted and renormalized to a geodesic ball.
More precisely, for center \(\mu_p\in\mathbb S^m\), concentration
\(\kappa\ge 0\), and truncation radius \(R_\ast \in (0, \pi]\), the density with
respect to Riemannian volume is
\[
f_{p,\kappa,R_\ast}(x)
=
\frac{
e^{\{\kappa\langle x,\mu_p\rangle\}}
\mathbf 1_{\{d(x,\mu_p)\le R_\ast\}}
}{
\operatorname{vol}(\mathbb S^{m-1})
\int_0^{R_\ast}e^{(\kappa\cos r)}(\sin r)^{m-1}\,dr
}.
\]
Thus, in geodesic polar coordinates \(x=\exp_{\mu_p}(r\theta)\) (N.B. $\exp_{\mu_p}$ is the Riemannian exponential at $\mu_p$), the
radial density is proportional to
\[
\exp(\kappa\cos r)(\sin r)^{m-1},\qquad 0\le r\le R_\ast.
\]
Note that when $R_\ast:=\pi,$ the above distribution reduces to the von Mises Fisher distribution.

\subsubsection{Local angular displacement and simulation design}
\label{subsec:numerical-design}

For each experiment, the null hypothesis is
\[
H_0:\ \bar\mu=N,
\]
where \(N\) is the north pole. The true population mean is shifted along a fixed geodesic by a scalar\textit{ angular parameter} \(p\), namely
\[
\mu_p=\exp_N(p e_1),
\]
where \(e_1\in T_N\mathbb S^m\) is a fixed unit vector. Thus \(p=0\) corresponds to the null hypothesis, while \(p\neq 0\) corresponds to a \textit{local} alternative. In the ambient representation, this is equivalently
\[
\mu_p=\cos(p)N+\sin(p)e_1.
\]

Samples were generated from a non-smeary, rotationally symmetric, truncated von Mises--Fisher type distribution centered at \(\mu_p\), with concentration parameter $\kappa,$ and support in the closed geodesic ball around \(\mu_p\). Since rotations of \(\mathbb S^m\) are isometries, rotating \(\mu_p\) to the north pole preserves the support radius \(R_\ast\) and rotational symmetry. Hence the support-threshold results cited below apply after this rotation and show that the generated law has unique population Fréchet mean \(\mu_p\); consequently \(p=0\) is the null case \(\bar\mu=N\), while \(p\ne0\) gives local alternatives.

In radial coordinates around \(\mu_p\), the radial density is proportional to (see Section \ref{sec:tvMF} above)
\[
e^{(\kappa\cos R)}(\sin R)^{m-1},
\qquad 0\le R\le R_\ast.
\]
In all experiments below we used
$$
R_\ast=1.0,\qquad \kappa=8.0,\qquad \alpha=0.05.
$$
Since \(R_\ast<R_m\) in the dimensions considered below, the experiments are in the \textit{non-smeary} regime (cf. Theorems 5.1 and 6.1 from \citet{Pal2026SharpSupportThresholds} that guarantee non-smeariness when the support is contained in a ball of radius $R_\ast$ in both the rotationally symmetric and general cases) covered by the Hessian estimator in Section~\ref{scn:modulation-aware-tests}, Equation \eqref{eqn:Hess-est-rot-symm} for rotationally symmetric case.

The \textit{naive quantile-based statistic} uses the empirical covariance of tangent vectors at the null mean and compares the resulting quadratic form with the \(\chi^2_m\) quantile. The \textit{bootstrap-based test} uses resampling under the null calibration, following the finite-sample-smeariness correction philosophy in \cite{EltznerFSSonspheres,HH2}. The \textit{modulation-aware test} uses the statistic \(T_n\) from Equation~\eqref{eqn:Gamma-n,T-n}, with the plug-in Hessian estimator \(H_n\) and the tangential covariance estimator \(\Sigma_n\) from Theorem~\ref{thm:mod-aware-test}.

\subsubsection{Low-dimensional experiment on \texorpdfstring{\(\mathbb S^2\)}{S2}}
\label{subsec:numerical-S2}

We first tested the behavior on \(\mathbb S^2\). The parameters were
\[
\text{(dimension) }m=2, \qquad \text{(sample size) }n=400,\qquad B=200,\qquad \text{number of repetitions}=400.
\]
Here \(B\) denotes the number of bootstrap resamples per simulated dataset. The angular grid was
\[
p\in\{-0.10,-0.03,-0.01,0,0.01,0.03,0.10\}.
\]
The empirical rejection probabilities are reported in Table~\ref{tab:S2-rejection} and plotted in Figure~\ref{fig:S2-rejection}.

\begin{table}[H]
\centering
\caption{Empirical rejection probabilities on \(\mathbb S^2\), with \(n=400\), \(400\) Monte Carlo repetitions, and \(B=200\) bootstrap resamples.}
\label{tab:S2-rejection}
\begin{tabular}{rrrr}
\toprule
\(p\) & Naive quantile-based & Bootstrap-based & Modulation-aware \\
\midrule
\(-0.100\) & \(1.000\) & \(1.000\) & \(1.000\) \\
\(-0.030\) & \(0.380\) & \(0.338\) & \(0.345\) \\
\(-0.010\) & \(0.068\) & \(0.062\) & \(0.062\) \\
\(0.000\)  & \(0.068\) & \(0.055\) & \(0.058\) \\
\(0.010\)  & \(0.092\) & \(0.075\) & \(0.070\) \\
\(0.030\)  & \(0.320\) & \(0.285\) & \(0.287\) \\
\(0.100\)  & \(1.000\) & \(1.000\) & \(1.000\) \\
\bottomrule
\end{tabular}
\end{table}

\begin{figure}[H]
\centering
\includegraphics[width=0.72\textwidth]{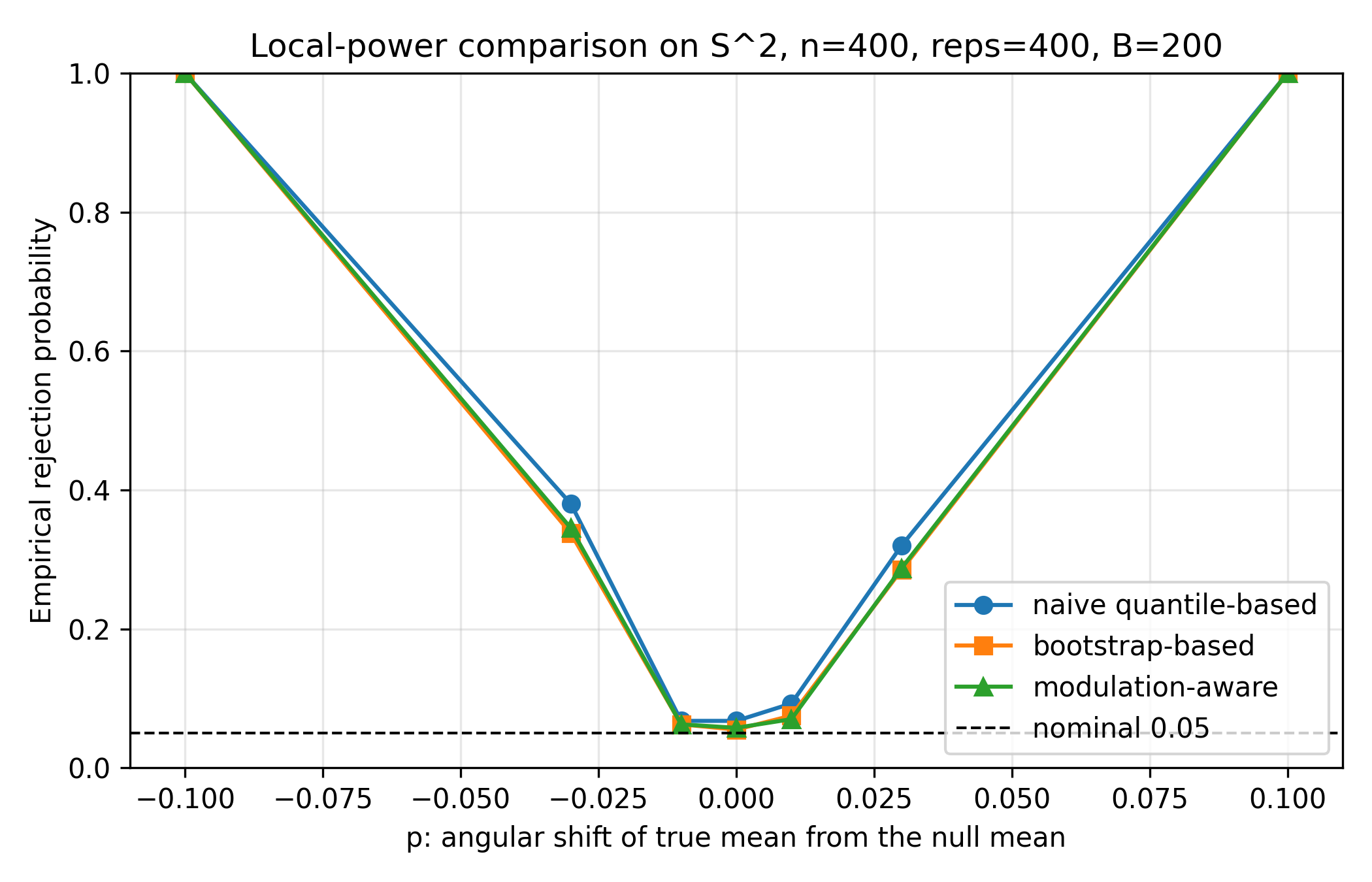}
\caption{Empirical rejection probabilities on \(\mathbb S^2\), with \(n=400\), \(400\) Monte Carlo repetitions, and \(B=200\) bootstrap resamples. The dashed horizontal line is the nominal level \(\alpha=0.05\).}
\label{fig:S2-rejection}
\end{figure}

In this low-dimensional experiment, all three tests behave reasonably near the null. At \(p=0\), the naive test rejects with probability \(0.068\), while the bootstrap-based and modulation-aware tests reject with probabilities \(0.055\) and \(0.058\), respectively. Thus the bootstrap and modulation-aware tests are slightly closer to the nominal level. For the local alternatives \(p=\pm 0.01\) and \(p=\pm 0.03\), the bootstrap-based and modulation-aware rejection probabilities are very close. In this regime, the main observation is therefore not a clear statistical separation between the two corrected tests, but rather that the Hessian-corrected statistic behaves comparably to bootstrap calibration.

\subsection{High-dimensional experiment on \texorpdfstring{\(\mathbb S^{36}\)}{S36}}
\label{subsec:numerical-S36}

We next tested the same procedure on \(\mathbb S^{36}\). The parameters were
\[
m=36,\qquad n=2500,\qquad B=30,\qquad \text{Monte Carlo repetitions}=100.
\]
The angular grid was concentrated near the null:
\[
p\in\{-0.010,-0.005,0,0.005,0.010\}.
\]
The empirical rejection probabilities are reported in Table~\ref{tab:S36-rejection} and plotted in Figure~\ref{fig:S36-rejection}.

\begin{table}[H]
\centering
\caption{Empirical rejection probabilities on \(\mathbb S^{36}\), with \(n=2500\), \(100\) Monte Carlo repetitions, and \(B=30\) bootstrap resamples.}
\label{tab:S36-rejection}
\begin{tabular}{rrrr}
\toprule
\(p\) & Naive quantile-based & Bootstrap-based & Modulation-aware \\
\midrule
\(-0.010\) & \(1.000\) & \(0.120\) & \(0.110\) \\
\(-0.005\) & \(0.950\) & \(0.080\) & \(0.070\) \\
\(0.000\)  & \(0.930\) & \(0.060\) & \(0.050\) \\
\(0.005\)  & \(0.970\) & \(0.090\) & \(0.070\) \\
\(0.010\)  & \(0.960\) & \(0.150\) & \(0.180\) \\
\bottomrule
\end{tabular}
\end{table}

\begin{figure}[H]
\centering
\includegraphics[width=0.72\textwidth]{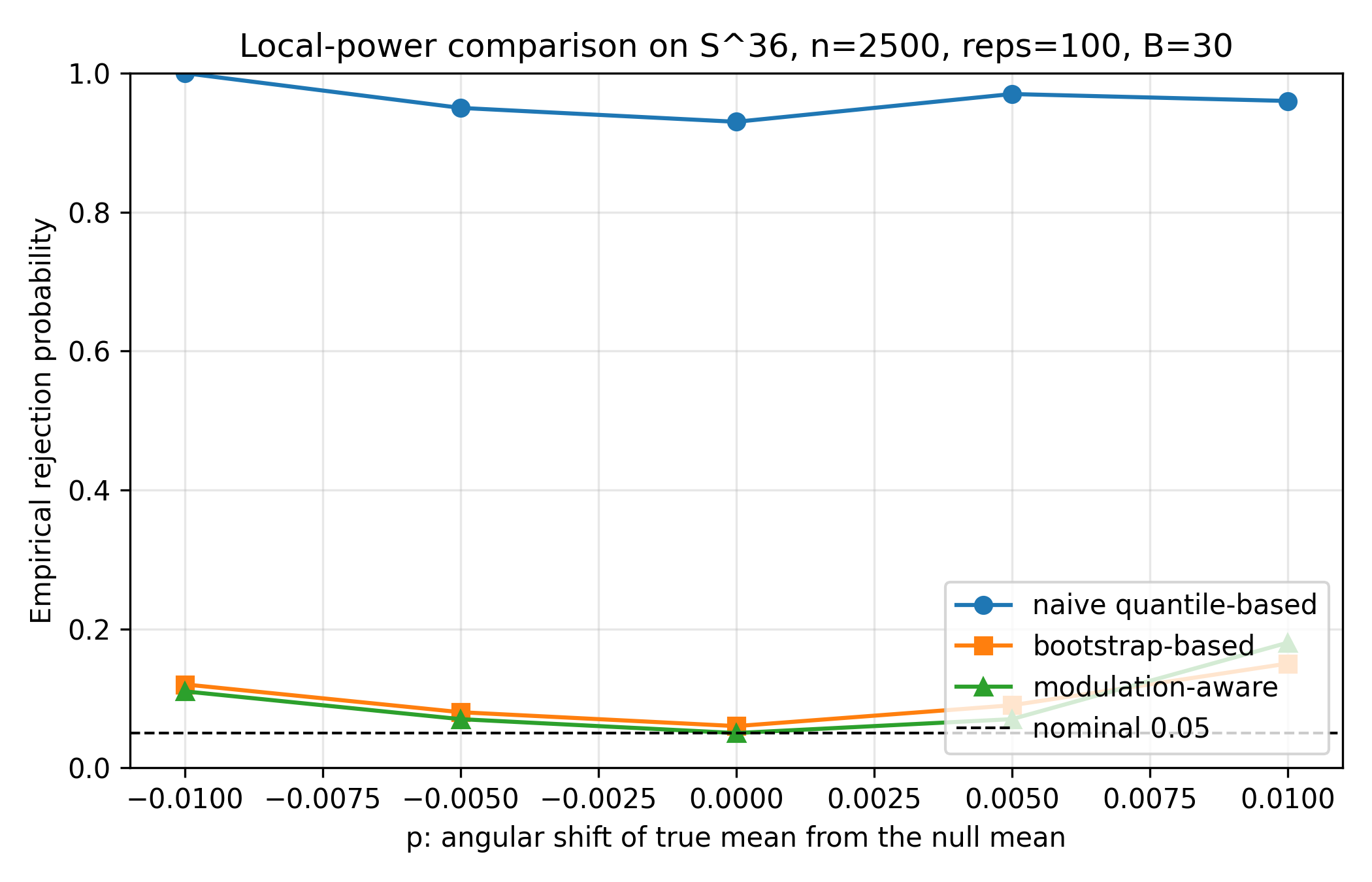}
\caption{Empirical rejection probabilities on \(\mathbb S^{36}\), with \(n=2500\), \(100\) Monte Carlo repetitions, and \(B=30\) bootstrap resamples. The dashed horizontal line is the nominal level \(\alpha=0.05\).}
\label{fig:S36-rejection}
\end{figure}

The high-dimensional behavior is substantially different from the \(\mathbb S^2\) case. The naive quantile-based test is severely miscalibrated: at the null \(p=0\), it rejects with probability \(0.930\), although the nominal level is \(0.05\). This is consistent with the curse of dimensionality for modulation studied in Theorem \ref{thm:dim_monotone_modulation_rot}.

Both corrected procedures substantially reduce this over-rejection. At \(p=0\), the bootstrap-based test rejects with probability \(0.060\), and the modulation-aware test rejects with probability \(0.050\). Thus, in this run, the modulation-aware test is exactly at the nominal level, while the bootstrap-based test is slightly above it. For \(p=\pm 0.005\), both corrected tests remain close to the nominal level. For \(p=\pm 0.010\), both tests begin to show local power, with rejection probabilities of the same order. The modulation-aware test is slightly lower in power at \(p=-0.010\), and higher at \(p=0.010\).

These numerical values should not be interpreted as proving that one corrected test uniformly dominates the other. The Monte Carlo standard error is non-negligible with \(100\) repetitions, especially for probabilities close to \(0.05\). The robust conclusion is that, in this high-dimensional experiment, the naive quantile-based test fails badly, whereas the bootstrap-based and modulation-aware tests both give a substantial correction. The modulation-aware test is comparable to bootstrap calibration in rejection behavior.

\subsubsection{Runtime comparison for the high-dimensional experiments}
\label{subsec:runtime-comparison}

Finally, we measured the runtime cost of the bootstrap-based and modulation-aware tests in the same high-dimensional setting. This timing experiment was performed separately from the rejection-probability simulation. It measures the average cost of one test call, not the runtime of the full simulation producing the curves above.

The timing parameters were
\[
m=36,\qquad n=2500,\qquad p=0,\qquad B=30,\qquad \alpha=0.05,
\]
with \(60\) timing repetitions. The results are given in Table~\ref{tab:runtime}.

\begin{table}[H]
\centering
\caption{Runtime comparison on \(\mathbb S^{36}\), with \(n=2500\), \(p=0\), and \(B=30\) bootstrap resamples. The reported values are per-test-call timings, averaged over \(60\) timing repetitions.}
\label{tab:runtime}
\begin{tabular}{lrr}
\toprule
 & Modulation-aware & Bootstrap-based \\
\midrule
Mean time per test call      & \(0.032608\) s & \(1.226873\) s \\
Median time per test call    & \(0.028348\) s & \(1.082563\) s \\
Total time over 60 calls     & \(1.956483\) s & \(73.612367\) s \\
\bottomrule
\end{tabular}
\end{table}

The ratio of mean runtimes was
\[
\frac{\text{mean bootstrap runtime}}{\text{mean modulation-aware runtime}}
=
37.62.
\]
Thus, in this implementation and hardware environment, the bootstrap-based test was approximately \(38\) times slower per test call, even with only \(B=30\) bootstrap resamples.

\begin{figure}[H]
\centering
\includegraphics[width=0.58\textwidth]{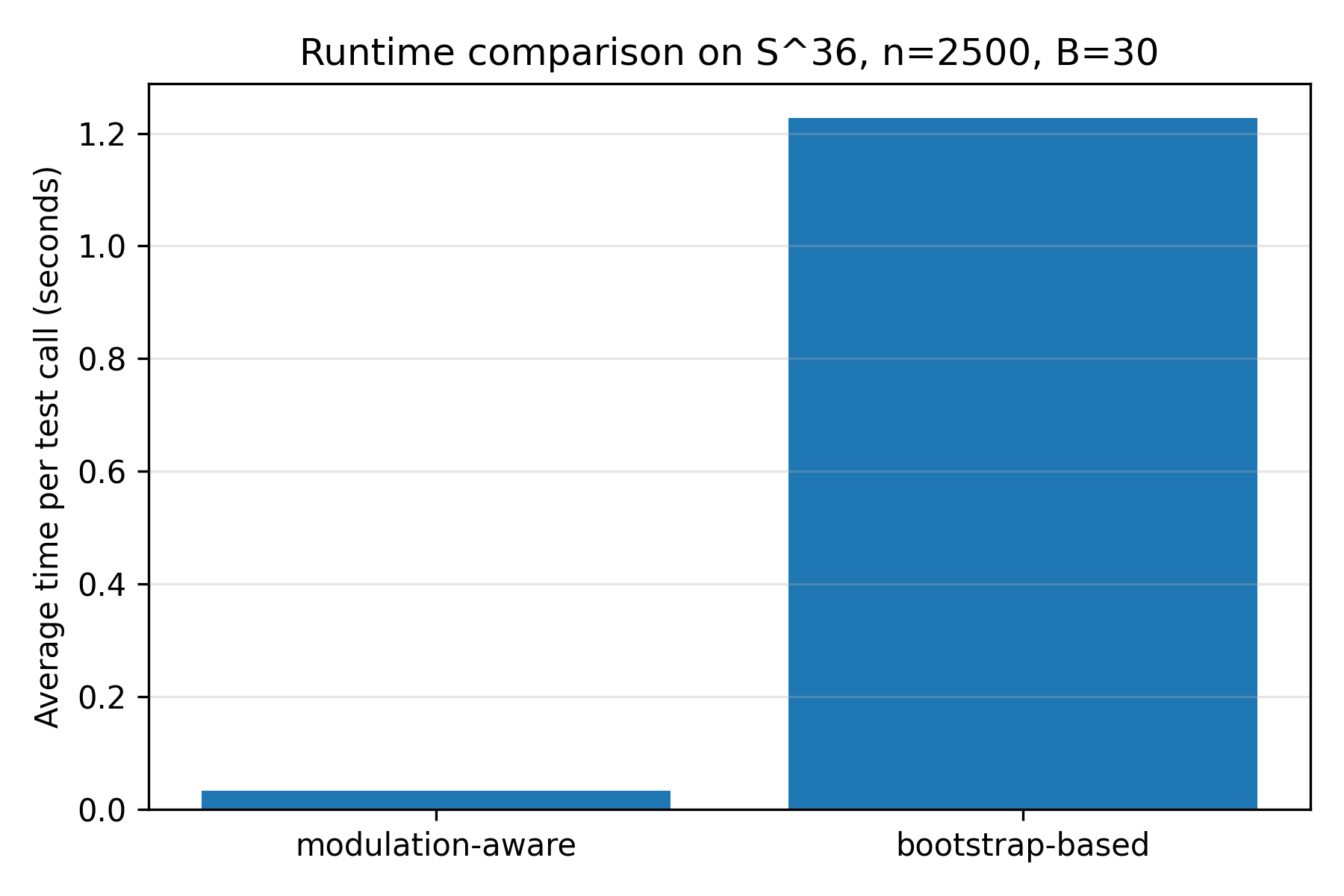}
\caption{Runtime comparison on \(\mathbb S^{36}\), with \(n=2500\), \(p=0\), and \(B=30\) bootstrap resamples. The plot reports average time per test call.}
\label{fig:runtime-comparison}
\end{figure}

These timings are implementation- and hardware-dependent. Nevertheless, the qualitative computational difference is expected. The bootstrap-based test repeatedly recomputes bootstrap Fréchet means and bootstrap test statistics, rendering it computationally expensive, see \citet{HH2}, Section 5.2. By contrast, the modulation-aware statistic uses one sample Fréchet mean together with the plug-in Hessian estimator \(H_n\) and tangential covariance estimator \(\Sigma_n\). Therefore the modulation-aware procedure avoids bootstrap calibration while retaining, in these experiments, rejection behavior comparable to the bootstrap-based correction.

\subsubsection{Summary of the numerical findings from synthetic dataset}
\label{subsec:numerical-summary}

The experiments support the following conclusions.

First, in low dimension, specifically on \(\mathbb S^2\), the bootstrap-based and modulation-aware tests have very similar rejection behavior. The naive test is only mildly more liberal near the null.

Second, in high dimension, specifically on \(\mathbb S^{36}\), the naive quantile-based test can be severely miscalibrated. In the experiment above, it rejects \(93\%\) of the time under the null. Both the bootstrap-based and modulation-aware tests substantially correct this behavior.

Third, the modulation-aware test gives bootstrap-comparable rejection probabilities in the tested regimes, while being much faster per test call. The numerical evidence therefore supports the practical value of the \textit{explicit Hessian correction}: it provides a direct non-bootstrap calibration of the quadratic form using the plug-in estimator \(H_n\) developed in Section~\ref{scn:modulation-aware-tests}.

\section{A real-data local-Euclidean diagnostic on \texorpdfstring{\(\mathbb S^{39}\)}{S39}}
\label{subsec:microtus-real-data-diagnostic}

In this section, we finally report a real-data diagnostic using landmark data from the
\textit{Microtus} lower first molar dataset of \citet{FoxVeneracionBlois2020}
and its accompanying Dryad archive \citep{FoxVeneracionBlois2020Dataset}.  The
dataset was designed to study measurement error in geometric morphometrics.
It contains two-dimensional landmark configurations of the lower first molar
from five North American \textit{Microtus} species, with \(21\) planar landmarks
per specimen.  The Dryad archive provides several repeated landmark datasets,
corresponding to different image-acquisition and digitization conditions.

Our goal in this subsection is \textit{not} to demonstrate superiority of the
modulation-aware statistic.  Instead, the aim is to check what the proposed
Hessian correction does on a real preshape-sphere dataset which turns out to be
highly concentrated around its Fr\'echet mean.  This gives a useful
``local-Euclidean'' sanity check: if the data occupy a small neighborhood of the
sphere, then the sphere should behave almost like its tangent space, the
empirical Hessian should be close to \(2I\), and the Hessian-corrected statistic
should reduce, up to a very small numerical difference, to the usual Hotelling
statistic.

We used the file
\[
\texttt{ProcANOVA/Nikon\_NoTilt\_EO\_T2.TPS}
\]
from the Dryad archive.  This gives \(N=247\) configurations, each consisting of
\(21\) two-dimensional landmarks.  Each configuration was centered by subtracting
its landmark centroid and then scaled to unit centroid size.  Thus each specimen
was represented as a preshape point on
\(
\mathbb S^{39}\subset \mathbb R^{40},
\)
since \(42\) raw coordinates are reduced by the two centering constraints and
then normalized to have unit norm.  We do not quotient by rotations here; the
purpose of this diagnostic is to work directly on the preshape sphere, where the
results of the preceding sections apply.  This is the standard preshape-sphere
representation used in landmark shape analysis; see, for example,
\citet{DrydenMardia2016}.

We then repeated the following random splitting procedure \(B=1000\) times.
In each repetition, we randomly divided the \(247\) observations into two disjoint groups with group sizes \(n_1=123\) and \(n_2=124\).  Since the partition is independent of the observed configurations, the artificial group labels introduce no systematic difference between the two groups. Thus this random-split experiment represents a null comparison within
the pooled dataset.

Let
\(\hat\mu_{1,b}\) and \(\hat\mu_{2,b}\) be the two\textit{ groupwise} sample Fr\'echet
means, and let \(\hat\mu_{0,b}\) be the\textit{ pooled} sample Fr\'echet mean.  We formed
the tangent-space difference
\[
D_b
:=
\log_{\hat\mu_{0,b}}(\hat\mu_{1,b})
-
\log_{\hat\mu_{0,b}}(\hat\mu_{2,b})
\in T_{\hat\mu_{0,b}}\mathbb S^{39}.
\]
Writing \(\widehat S_{j,b}\) for the empirical tangent covariance matrix of
\(\log_{\hat\mu_{0,b}}(X_{j,i})\) in group \(j\), an estimator of the covariance of the difference between the two sample mean vectors, namely \(D_b\) is
\[
\widehat C^{\mathrm{naive}}_b
:=
\frac{\widehat S_{1,b}}{n_1}
+
\frac{\widehat S_{2,b}}{n_2}.
\]
The corresponding naive two-sample Hotelling statistic is
\[
T^{\mathrm{naive}}_b
:=
D_b^T
\bigl(\widehat C^{\mathrm{naive}}_b\bigr)^{\dagger}
D_b,
\]
where \(A^\dagger\) denotes the Moore--Penrose inverse.  In the present
experiment all reported covariance ranks were \(39\), so the use of the
pseudoinverse is only a numerical safeguard.

For the \textit{modulation-aware} statistic, we used the groupwise plug-in Hessian
estimators \(\widehat H_{1,b}\) and \(\widehat H_{2,b}\) at the pooled base point
\(\hat\mu_{0,b}\).  The \textit{Hessian-corrected covariance estimator} is
\[
\widehat C^{\mathrm{mod}}_b
:=
\frac{
4\widehat H_{1,b}^{-1}\widehat S_{1,b}\widehat H_{1,b}^{-1}
}{n_1}
+
\frac{
4\widehat H_{2,b}^{-1}\widehat S_{2,b}\widehat H_{2,b}^{-1}
}{n_2}.
\]
The modulation-aware statistic is then
\[
T^{\mathrm{mod}}_b
:=
D_b^T
\bigl(\widehat C^{\mathrm{mod}}_b\bigr)^{\dagger}
D_b.
\]
This is the two-sample analogue of the Hessian-corrected statistic from
Theorem~\ref{thm:two-sample-mod-aware-hotelling}.

There is one additional finite-sample point which is important in this real-data
experiment.  The tangent dimension is \(39\), while the total sample size is only
\(N= n_1+ n_2= 247\).  Therefore the crude large-sample \(\chi^2_{39}\) calibration is not
expected to be accurate for an ordinary Hotelling statistic, because the
covariance matrix is estimated from the data.  To avoid confusing this classical
finite-sample Hotelling effect with the geometric Hessian correction, we also
used the usual Hotelling-style \(F\)-calibration
\[
F_b
=
\frac{N-r-1}{r(N-2)}\,T_b = \frac{n_1+n_2-r-1}{r(n_1+n_2-2)}\,T_b,
\qquad
F_b\ \text{compared with}\ F_{r,N-r-1},
\]
where \(r=39\) is the numerical rank of the covariance matrix.  Thus here
\(F_b\) is compared with an \(F_{39,207}\) distribution.  For the \textit{naive} statistic
this is the standard finite-sample Hotelling calibration in the Euclidean
model.  For the Hessian-corrected statistic we use the same calibration only as
a finite-sample diagnostic; the asymptotic result proved in
Theorem~\ref{thm:two-sample-mod-aware-hotelling} remains the \(\chi^2\)
calibration.

For a nominal level \(\alpha\), let \(p^{\mathrm{naive}}_{b,F}\) and
\(p^{\mathrm{mod}}_{b,F}\) denote the upper-tail \(p\)-values obtained in the
\(b\)-th random split from the \(F_{39,207}\)-calibrations of the naive and
modulation-aware statistics, respectively; explicitly,
\(p^{\mathrm{naive}}_{b,F}
=\Pr\!\left(F_{39,207}\ge \frac{207}{39\cdot245}T_b^{\mathrm{naive}}\right)\),
and analogously for \(p^{\mathrm{mod}}_{b,F}\) with
\(T_b^{\mathrm{mod}}\).  Likewise, let
\(p^{\mathrm{naive}}_{b,\chi^2}\) and \(p^{\mathrm{mod}}_{b,\chi^2}\)
denote the corresponding \(p\)-values obtained from the
\(\chi^2_{39}\)-calibrations.  The four empirical rejection rates are
\[
\begin{aligned}
\widehat\alpha_{\mathrm{naive},F}(\alpha)
&=
\frac1B\sum_{b=1}^B
\mathbf 1\{p^{\mathrm{naive}}_{b,F}<\alpha\},
&
\widehat\alpha_{\mathrm{mod},F}(\alpha)
&=
\frac1B\sum_{b=1}^B
\mathbf 1\{p^{\mathrm{mod}}_{b,F}<\alpha\},\\
\widehat\alpha_{\mathrm{naive},\chi^2}(\alpha)
&=
\frac1B\sum_{b=1}^B
\mathbf 1\{p^{\mathrm{naive}}_{b,\chi^2}<\alpha\},
&
\widehat\alpha_{\mathrm{mod},\chi^2}(\alpha)
&=
\frac1B\sum_{b=1}^B
\mathbf 1\{p^{\mathrm{mod}}_{b,\chi^2}<\alpha\}.
\end{aligned}
\]
The results are given in
Table~\ref{tab:microtus-real-data-rejection}.  The \(F\)-calibrated rejection
rates are close to their nominal levels for both statistics.  By contrast, the
large-sample \(\chi^2\) calibration is much too liberal in this finite-sample
setting.

\begin{table}[H]
\centering
\caption{Random null-split calibration on the \textit{Microtus}
preshape-sphere dataset.  The experiment used \(B=1000\) random splits of
one measurement condition, with \(N=247\) observations on
\(\mathbb S^{39}\).  The \(F\)-calibrated columns are the main finite-sample
diagnostic.  The \(\chi^2\)-calibrated columns are shown only to illustrate
the over-rejection caused by using the large-sample Hotelling approximation
in this dimension and sample size.}
\label{tab:microtus-real-data-rejection}
\begin{adjustbox}{max width=\textwidth}
\begin{tabular}{rrrrr}
\toprule
Nominal level \(\alpha\)
& \(\widehat\alpha_{\mathrm{naive},F}(\alpha)\)
& \(\widehat\alpha_{\mathrm{mod},F}(\alpha)\)
& \(\widehat\alpha_{\mathrm{naive},\chi^2}(\alpha)\)
& \(\widehat\alpha_{\mathrm{mod},\chi^2}(\alpha)\) \\
\midrule
\(0.01\) & \(0.011\) & \(0.010\) & \(0.094\) & \(0.087\) \\
\(0.03\) & \(0.029\) & \(0.027\) & \(0.175\) & \(0.171\) \\
\(0.05\) & \(0.045\) & \(0.044\) & \(0.224\) & \(0.216\) \\
\(0.07\) & \(0.072\) & \(0.065\) & \(0.270\) & \(0.268\) \\
\(0.10\) & \(0.101\) & \(0.096\) & \(0.320\) & \(0.317\) \\
\bottomrule
\end{tabular}
\end{adjustbox}
\end{table}

We next explain why the naive and modulation-aware columns are almost identical.
Let
\[
Z_i:=\log_{\hat\mu_0}(X_i)\in T_{\hat\mu_0}\mathbb S^{39}
\]
denote the tangent vectors at the full-sample Fr\'echet mean.  The average
geodesic radius of the data cloud is small:
\[
\frac1N\sum_{i=1}^N d(X_i,\hat\mu_0)
=
0.084761\ \text{radians}
=
4.856^\circ,
\]
and even the largest observed radius is only
\[
\max_i d(X_i,\hat\mu_0)
=
0.279323\ \text{radians}
=
16.004^\circ.
\]
The total tangent variance is also small:
\[
\operatorname{tr}(\widehat\Sigma)=0.008423.
\]
Thus the observed landmark configurations occupy a small region of the preshape
sphere.  In such a local region, the sphere is well approximated by its tangent
space.

This concentration is also visible directly in the Hessian.  The full-sample
empirical Hessian satisfies
\[
\lambda_{\min}(\widehat H)=1.994404,\qquad
\lambda_{\mathrm{mean}}(\widehat H)=1.994544,\qquad
\lambda_{\max}(\widehat H)=1.997522.
\]
Hence \(\widehat H\) is extremely close to \(2I\):
\[
\frac{\|\widehat H-2I\|_F}{\|2I\|_F}
=
0.002739.
\]

This is exactly the local-Euclidean behavior predicted by the general plug-in
Hessian estimator in Equation~\eqref{eqn:Hess-est-gen}, which is used here.
Applied to the full sample at the base point \(\hat\mu_0\), that formula has
\(m=39\), \(n=N\), and \(R_{i,N}=d(\hat\mu_0,X_i)\).  Since
\(R_{i,N}\cot R_{i,N}
=1-R_{i,N}^2/3+O(R_{i,N}^4)\) as \(R_{i,N}\to0\), the difference between the
\(i\)th bracketed term in Equation~\eqref{eqn:Hess-est-gen} and \(I_{39}\) is
\(O(R_{i,N}^2)\). Consequently, as \(\max_i R_{i,N}\to0\),
\(\|\widehat H-2I_{39}\|_F
=O\!\left(N^{-1}\sum_{i=1}^N R_{i,N}^2\right)\).
The small observed radii therefore directly explain why \(\widehat H\) is close
to the Euclidean value \(2I_{39}\).

The numerical effect on the covariance estimator is consequently very small.
Across the \(1000\) random splits,
\[
\operatorname{median}_{b}
\frac{
\|\widehat C^{\mathrm{mod}}_b-\widehat C^{\mathrm{naive}}_b\|_F
}{
\|\widehat C^{\mathrm{naive}}_b\|_F
}
=
0.002680.
\]
Thus the Hessian-corrected covariance differs from the naive covariance by only
about \(0.27\%\) in relative Frobenius norm.  The main local-Euclidean
diagnostics are summarized in Table~\ref{tab:microtus-local-euclidean}.

\begin{table}[H]
\centering
\caption{Local-Euclidean diagnostics for the \textit{Microtus} preshape-sphere dataset.  The small geodesic radii show that the data are concentrated near their Fr\'echet mean.  The Hessian eigenvalues are almost \(2\), and the relative change from the naive covariance to the modulation-aware covariance is about \(0.27\%\).}
\label{tab:microtus-local-euclidean}
\begin{adjustbox}{max width=\textwidth}
\begin{tabular}{lr}
\toprule
Diagnostic quantity & Value \\
\midrule
Mean geodesic radius \(\frac1N\sum_i d(X_i,\hat\mu_0)\) & \(0.084761\) rad \((4.856^\circ)\) \\
Median geodesic radius & \(0.078604\) rad \((4.504^\circ)\) \\
Maximum geodesic radius & \(0.279323\) rad \((16.004^\circ)\) \\
Mean squared geodesic radius \(\frac1N\sum_i d^2(X_i,\hat\mu_0)\) & \(0.008389\) \\
Trace of tangent covariance \(\operatorname{tr}(\widehat\Sigma)\) & \(0.008423\) \\
Smallest covariance eigenvalue \(\lambda_{\min}(\widehat\Sigma)\) & \(0.000006\) \\
Largest covariance eigenvalue \(\lambda_{\max}(\widehat\Sigma)\) & \(0.004692\) \\
Smallest Hessian eigenvalue \(\lambda_{\min}(\widehat H)\) & \(1.994404\) \\
Mean Hessian eigenvalue \(\lambda_{\mathrm{mean}}(\widehat H)\) & \(1.994544\) \\
Largest Hessian eigenvalue \(\lambda_{\max}(\widehat H)\) & \(1.997522\) \\
Relative Hessian deviation \(\|\widehat H-2I\|_F/\|2I\|_F\) & \(0.002739\) \\
Median relative covariance change
\(\operatorname{median}_b\|\widehat C_b^{\mathrm{mod}}-\widehat C_b^{\mathrm{naive}}\|_F/
\|\widehat C_b^{\mathrm{naive}}\|_F\)
& \(0.002680\) \\
\bottomrule
\end{tabular}
\end{adjustbox}
\end{table}

The conclusion is therefore deliberately modest.  This real-data experiment does
not show a practical separation between the naive and modulation-aware
statistics.  Rather, it shows that on this \textit{Microtus} measurement
condition the data are so concentrated near the Fr\'echet mean that the
preshape sphere is effectively locally Euclidean.  In this regime,
\[
\widehat H\approx 2I,
\qquad
4\widehat H^{-1}\widehat\Sigma\widehat H^{-1}\approx\widehat\Sigma,
\]
and hence the Hessian-corrected statistic reduces numerically to the usual
Hotelling statistic.  This is consistent with the geometric interpretation of
the modulation correction: it matters when curvature changes the covariance of
the Fr\'echet mean, and it disappears when the data lie in a small region where
the manifold behaves almost Euclidean.

\appendix
\section{Symbolic Computation Code}\label{scn:symbolic-computation}

The following Python (SymPy) code was used to compute the Taylor coefficients:

\begin{verbatim}
import sympy as sp

# Symbols
t, a, R, alpha = sp.symbols('t a R alpha', real=True)
r = t*a

X = (sp.cos(r)*sp.cos(R)
     + (sp.sin(r)/r)*(sp.sin(R)/R)*(t*alpha))

f = sp.acos(X)**2

series_f = sp.series(f, t, 0, 5)
series_noO = sp.expand(series_f.removeO())

# Extract coefficients
coeffs = {}
for k in range(5):
    ck = sp.simplify(series_noO.coeff(t, k))
    coeffs[k] = ck
    print(f"C{k} =", ck)
\end{verbatim}

\bibliographystyle{plainnat}   
\bibliography{references}      

\end{document}